\documentclass[12pt]{article}
\usepackage{srcltx}
\usepackage{amsfonts,amssymb,mathrsfs,amsmath,cmtiup}
\usepackage{amsthm}

\newtheorem{theorem}{Theorem}[section]
\newtheorem{lemma}[theorem]{Lemma}
\newtheorem{proposition}[theorem]{Proposition}
\newtheorem{corollary}[theorem]{Corollary}

\theoremstyle{definition}

\newtheorem*{Index Convention}{Index Convention}

\def\keywords#1{\par\medskip
\noindent\textbf{Keywords.} #1}

\def\subjclass#1{{\renewcommand{\thefootnote}{}
\footnote{\emph{Mathematics Subject Classification (2010):} #1}}}

\begin{document}
\let\le=\leqslant
\let\ge=\geqslant
\let\leq=\leqslant
\let\geq=\geqslant
\newcommand{\e}{\varepsilon }
\newcommand{\f}{\varphi }
\newcommand{ \g}{\gamma}
\newcommand{\F}{{\Bbb F}}
\newcommand{\N}{{\Bbb N}}
\newcommand{\Z}{{\Bbb Z}}
\newcommand{\Q}{{\Bbb Q}}
\newcommand{\C}{{\Bbb C}}
\newcommand{\R}{\Rightarrow }
\newcommand{\W}{\Omega }
\newcommand{\w}{\omega }
\newcommand{\s}{\sigma }
\newcommand{\hs}{\hskip0.2ex }
\newcommand{\ep}{\makebox[1em]{}\nobreak\hfill $\square$\vskip2ex }
\newcommand{\Lr}{\Leftrightarrow }
\sloppy


\title{Lie type algebras with an automorphism of finite order}

\markright{}

\author{{\sc N.\,Yu.~Makarenko \\ \small Sobolev Institute of Mathematics, Novosibirsk,
630\,090,
Russia, \\[-1ex] \small natalia\_makarenko@yahoo.fr
}}

\author{
{N.\,Yu.~Makarenko\footnote{The work is supported by  Russian
Science Foundation, project 14-21-00065.}}\\
\small  Sobolev Institute of Mathematics, Novosibirsk, 630\,090,
Russia
\\[-1ex] \small  natalia\_makarenko@yahoo.fr}

\date{}
\maketitle



\begin{abstract} An algebra $L$ over a field $\Bbb
F$, in which product is denoted by $[\,,\,]$, is said to be
\textit{ Lie type algebra}  if for all elements $a,b,c\in L$ there
exist $\alpha, \beta\in \Bbb F$ such that $\alpha\neq 0$ and
$[[a,b],c]=\alpha [a,[b,c]]+\beta[[a,c],b]$.  Examples of Lie type
algebras are associative algebras, Lie algebras, Leibniz algebras,
etc.   It is proved that if a Lie type algebra $L$ admits an
automorphism of finite order $n$ with finite-dimensional
fixed-point subalgebra of dimension $m$, then $L$ has a soluble
ideal  of finite codimension bounded in terms of $n$ and $m$ and
of derived length bounded in terms of $n$.
\end{abstract}


\section{Introduction}
 By Kreknin's theorem~\cite{kr} a Lie algebra over a field admitting a
fixed-point-free automorphism of finite order $n$ is soluble of
derived length at most~$\leq 2^n-2$. In~\cite{khu-ma,ma-khu} it
was proved that a Lie algebra with an ``almost regular''
automorphism of finite order is almost soluble: if a Lie algebra
$L$ over a field admits an automorphism $\varphi$ of finite order
$n$ such that the fixed-point subalgebra $C_L(\varphi )$ has
finite dimension $m$, then $L$ has a soluble ideal of finite
codimension bounded in terms of $n$ and $m$ and of derived length
bounded in terms of $n$.

The proofs of the above results are purely combinatorial  and do
not use the structure theory. This fact makes it possible to
extend them to a broader class of algebras including associative
algebras, Lie algebras,  Leibniz algebras and others. Throughout
the present paper, a \textit{Lie type algebra} means an algebra
$L$ over a field $\Bbb F$ with product $[\,,\,]$ satisfying the
following property: for all elements $a, b, c \in L$ there exist
$\alpha, \beta \in \Bbb F$ such that $\alpha\ne 0$ and
$$
[[a,b],c]=\alpha [a,[b,c]]+\beta[[a,c],b].$$ Note that in general
$\alpha, \beta$ depend on elements $a,b,c\in L$;  they can be
viewed as functions $\alpha, \beta: L\times L \times L \rightarrow
\Bbb F$.

 The main result of the paper is the following

 \begin{theorem}\label{th2}
 Suppose that a Lie type  algebra $L$ \/$($of  possibly infinite dimension\/$)$ over an arbitrary field admits an
automorphism of finite order $n$ with finite-dimensional
fixed-point subalgebra of dimension $m$, then $L$ has a soluble
ideal  of finite codimension bounded in terms of $n$ and $m$ and
of derived length bounded in terms of $n$.
\end{theorem}

Theorem \ref{th2} is also non-trivial for finite-dimensional Lie
type algebras because of the bound for the codimension. Note that
no results of this kind is possible for an automorphism of {\it
infinite\/} order: a free Lie algebra on the free generators
$f_i$, $i\in \Bbb Z $, admits the regular automorphism given by
the mapping $f_i\rightarrow f_{i+1}$. The proof reduces to
considering a $({\Bbb Z} /n{\Bbb Z} )$-graded algebras with
finite-dimensional zero component (Theorem~\ref{th1}). Recall that
an algebra $L$ over a field $\Bbb F$ with product $[\,,\,]$ is
\textit{ $({\Bbb Z}/n{\Bbb Z})$-graded} if
$$L=\bigoplus_{i=0}^{n-1}L_i\qquad \text{ and }\qquad [L_i,L_j]\subseteq L_{i+j\,({\rm mod}\,n)},$$
where  $L_i$ are subspaces  of~$L$. Elements of $L_i$ are referred
to as homogeneous and the subspaces $L_i$  are called homogeneous
components or  grading components. In particular, $L_0$ is called
the zero component or the identity component.

Finite cyclic gradings naturally arise in
 the study of algebras admitting  an automorphism of finite order.
This is due to the  fact that, after the ground field is  extended
by a primitive $n$th root of unity~$\omega$, the eigenspaces
$L_j=\{ a \mid \varphi (a)=\omega ^ja\}$ behave like the
components of a $({\Bbb Z} /n{\Bbb Z})$-grading:
$[L_s,L_t]\subseteq L_{s+t},$ where $s+t$ is calculated modulo
$n$.  For example,  Kreknin's theorem~\cite{kr} can be
reformulated in terms of graded Lie algebras  as follows: a
$({\Bbb Z}/n{\Bbb Z})$-graded Lie algebra
$L=\bigoplus_{i=0}^{n-1}L_i$ over an arbitrary field with trivial
zero component $L_0=0$ is soluble of derived length at most~$\leq
2^n-2$. The proof of the result on ``almost regular''
automorphisms in~\cite{ma-khu}  also reduces to considering a
$({\Bbb Z}/n{\Bbb Z})$-graded Lie algebra
$L=\bigoplus_{i=0}^{n-1}L_i$, but in this case the zero component
$L_0$ has finite dimension~$m$.

Before stating Theorem \ref{th1}, we introduce a notion of a
\textit{$({\Bbb Z} /n{\Bbb Z} )$-graded Lie type algebra} as a
$({\Bbb Z} /n{\Bbb Z} )$-graded algebra $L=
\bigoplus_{i=0}^{n-1}L_i$ over a field $\Bbb F$  with product
$[\,,\,]$ satisfying the following property: for all
\textit{homogeneous} elements $a, b, c \in L$ there exist $\alpha,
\beta \in \Bbb F$ such that $\alpha\ne 0$ and
\begin{equation}\label{tozh}
[[a,b],c]=\alpha [a,[b,c]]+\beta[[a,c],b]. \end{equation} The only
difference with the definition of a Lie type algebra is that
property (\ref{tozh}) is defined only for homogeneous elements of
$L$. It is clear that if $n=1$, then a $({\Bbb Z} /n{\Bbb Z}
)$-graded Lie type algebra is a Lie type algebra. An important
example of a $({\Bbb Z} /n{\Bbb Z} )$-graded Lie type algebra
(that is  not a Lie type algebra) is a color Lie superalgebra.

\vskip1ex \noindent {\bf Remark.}  Our class of $({\Bbb Z} /n{\Bbb
Z} )$-graded Lie type algebras includes  algebras of Lie type in
the sense of Bakhturin-Zaicev  introduced in \cite{ba-za}.

\vskip1ex

 In \cite{brgr} Bergen and Grzeszczuk  extended
Kreknin's theorem~\cite{kr} to $({\Bbb Z} /n{\Bbb Z} )$-graded Lie
type algebras. They established (even in a more general setting of
so called $(\alpha, \beta, \gamma)$-algebra) the solubility of  a
$({\Bbb Z} /n{\Bbb Z} )$-graded Lie type algebra $L=
\bigoplus_{i=0}^{n-1}L_i$  with trivial zero component $L_0=0$.

The following theorem  deals with  case of $\mathrm{dim}\, L_0=m$
and extends the above mentioned result of \cite{ma-khu}  to
$({\Bbb Z} /n{\Bbb Z} )$-graded Lie type algebras.

 \begin{theorem}\label{th1} Let $n$ be a positive integer and  $L= \bigoplus_{i=0}^{n-1}L_i$  a $({\Bbb Z} /n{\Bbb Z} )$-graded
\/$($possibly infinite-dimensional\/$)$  Lie type algebra  over an
arbitrary field. If the zero component $L_0$ has finite dimension
$m$, then $L$ has a homogeneous soluble ideal of finite
codimension bounded in terms of $n$ and $m$ and of derived length
bounded in terms of $n$. In the particular case of $m=0$, the
algebra $L$ is soluble of derived length bounded in terms of $n$.
\end{theorem}

Theorem \ref{th1} implies Theorem \ref{th2}, but it also has an
independent interest of its own; in particular, Theorems \ref{th5}
and \ref{th6} on color Lie superalgebras follow from it
(see~\S\,2).

\vskip1ex

 A \textit{\/$($right\/$)$ Leibniz
algebra} or  \textit{Loday algebra} is an algebra $L$ over a field
 with  bilinear product $[\, ,\, ]$ satisfying the Leibniz
identity $$[[a,b],c] = [a,[b,c]]+ [[a,c],b]$$ for all $a,b,c\in
L$.

 If under the hypothesis of Theorems \ref{th2}  we
set $\alpha=1$, $\beta=1$ for all $a, b, c \in L$, then the
algebra $L$ becomes a (right) Leibniz algebra and we immediately
get the following corollaries.

\vskip1ex

In what follows, we use abbreviation, say, ``$(m,n,\dots
)$-bounded'' for ``bounded above in terms of  $m, n,\dots$.

\begin{corollary}\label{th4} If a Leibniz algebra $M$ admits an
automorphism $\varphi$ of finite order~$n$ with finite-dimensional
fixed-point subalgebra of dimension $m$, then $M$ has a soluble
ideal of $n$-bounded derived length and of finite $(n,m)$-bounded
codimension. If $m=0$ then $M$ is soluble  of $n$-bounded derived
length.
\end{corollary}

\begin{corollary}\label{th3} Let $n$ be a positive integer and  $L= \bigoplus_{i=0}^{n-1}L_i$  a $({\Bbb Z} /n{\Bbb Z} )$-graded
 Leibniz algebra over an arbitrary field. If the zero component  $L_0$ has
finite dimension $m$, then $L$ has a homogeneous soluble ideal of
$n$-bounded derived length and of finite $(n,m)$-bounded
codimension. If $m=0$, then $M$ is soluble  of $n$-bounded derived
length.
\end{corollary}

The proof of Theorem~\ref{th1} follows the same scheme as that
of~\cite{ma-khu}. Combinatorial arguments in~\cite{ma-khu} are
based only on the Jacoby identity and the  anticommutativity
identity in Lie algebras. In our case the Jacoby identity can be
successfully replaced by  property~(\ref{tozh}). The main
difficulty facing us is  the lack of the anticommutativity. In
order to manage this complication we had to somewhat change the
principal construction and to re-prove all the lemmas.   For the
reader's convenience we give detailed proofs of all lemmas even
though some of them overlap significantly with the proofs of
analogous lemmas in \cite{ma-khu}.

The core of the proof is the method of generalized centralizers
created by Khukhro for Lie rings and nilpotent groups with almost
regular automorphisms of prime order \cite{kh1} and developed
further in \cite{ma-khu96, ma-khu98, khu-ma, ma-khu}. The
sought-for ideal $Z$ is generated by so called generalized
centralizers $L_i(N)$, certain subspaces of the homogeneous
components $L_i$, $i\ne 0$, of finite $(n,m)$-bounded
codimensions. Construction of generalized centralizers $L_i(k)$ of
increasing levels $k=1,\ldots, N$ is realized by induction  up to
some $n$-bounded value~$N$. Simultaneously, certain elements
$z_i(k)$, called representatives of level $k$, are fixed. Elements
of the $L_j(k)$ have a centralizer property with respect to the
representatives of lower levels: if a product of bounded length
involves exactly one element $y_j\in L_j(k)$ of level $k$ and some
representatives $z_i(s)\in L_i(s)$ of lower levels $s<k$ and
belongs to $L_0$, then this product is equal to~$0$. The proof of
the fact that $Z$ is soluble of bounded derived length is based on
Proposition \ref{p1} which is an analogue of solubility criterion
in~\cite{kh3}  for Lie rings. Proposition \ref{p1} reduces the
solubility of $Z$ to the solubility of the subalgebra generated by
the subspace
$$S(Z)=\sum_{t=0}^{T}[\underbrace{Z_0,\ldots, Z_0}_t, Z , \underbrace{Z_0,\ldots, Z_0}_{T-t}],$$ where $Z_0=Z\cap L_0$ and $T$
is a certain $n$-bounded number.
 It is applied repeatedly to the series of embedded subalgebras $Z\langle i\rangle$ of
 $Z$ constructed inductively as follows: $Z\langle 1\rangle =Z$; $Z\langle i+1\rangle $ is the subalgebra generated by
$S (Z\langle i \rangle)$.  Thus the proof boils down to the fact
that $Z_0 \langle Q \rangle$ is trivial for  some  $n$-bounded
number $Q$. This is
 accomplished by intricate and subtle calculations by means of  $zc$-elements,
 some special elements of $L_0$ of increasing complexity
which, in particular, generate $Z_0\langle
 i\rangle$.

The paper is organized as follows. Corollaries for Lie
superalgebras and color Lie superalgebras are presented in~\S\,2.
We introduce some definitions and notations in~\S\,3. Then we
prove in \S\,4 the solubility criterion (Proposition~\ref{p1}).
Generalized centralizers and fixed representatives are constructed
and their basic properties are listed in~\S\,5.
 In \S\,6 we construct the  required soluble ideal $Z$  and define $zc$-elements. In
 \S\,7
we establish the basic properties of $zc$-elements. In \S\,8
Theorem~\ref{th1} is proved. In \S\,9 we determine the scheme of
choice of the parameters. In~\S\,10 we prove Theorem~\ref{th2} for
``almost regular'' automorphisms and derive results for color Lie
superalgebras.

\section{Corollaries for color Lie superalgebras}

Before stating corollaries  of Theorem~\ref{th1} for Lie
superalgebras and  color Lie superalgebras we recall some
definitions.

 \vskip1ex

A $(\Bbb Z /2\Bbb Z )$-graded algebra $L=L_0+L_1$  with
multiplication  $[\,,\,]$ is called  {\it Lie superalgebra} if
$$[a,b]=-(-1)^{\alpha\beta}[b,a]$$ and
$$[a,[b,c]]=[[a,b],c]+(-1)^{\alpha\beta}[b,[a,c]]$$
 for   $a\in L_{\alpha},\,\, b\in L_{\beta}$.

\vskip1ex

Let  $Q$ be an abelian group. A $Q$-graded algebra $L=\oplus_{q\in
Q} L_q$ is called {\it color Lie superalgebra} if for all
homogeneous elements $x\in L_p$, $y\in L_q$, $z\in L_t$ the
following equations hold:
$$xy=-\epsilon(p,q)yx,$$
$$x(yz)=(xy)z+\epsilon(p,q)y(xz),$$
where  $\epsilon(p,q)$ is a skew-symmetric bilinear form, that is
$\epsilon: Q\times Q\rightarrow F^*,$
$$\epsilon(p,q)\epsilon(q,p)=1,\,\,\,
\epsilon(p_1+p_2,q)=\epsilon(p_1,q)\epsilon(p_2,q) \,\,\, \text{
and }\,\,\, \epsilon(p,q_1+q_2)=\epsilon(p,q_1)\epsilon(p,q_2).$$

\vskip1ex

Let $G$ be an abelian group written multiplicatively. We say that
a color Lie superalgebra $L$  is $G$-graded (or  has a
$G$-grading) if $L$ is a direct sum of spaces $L^{(g)}$:
$$L=\bigoplus_{g\in G} L^{(g)},$$ such that $[L^{(g)}, L^{(h)}]\subset
L^{(gh)}$ and   $L^{(g)}$ are homogeneous with respect to the
$Q$-grading, that is
$$L^{(g)}=\bigoplus_{q\in Q} (L^{(g)}\cap L_q).$$
We will denote the subspace $L^{(g)}\cap L_q$ by $L_q^{(g)}$ and
the neutral element of $G$ by $e$. Then $L^{(e)}$ is the
homogeneous component corresponding to the neutral element $e\in
G$  and, consequently, $L_0^{(e)}=L^{(e)}\cap L_0$.

\vskip1ex

  The following results are almost straightforward
consequences of Theorem \ref{th1}.

\begin{theorem}\label{th5}
 Let  $Q$ and  $G$ be finite cyclic groups of coprime orders  $k$
 and $n$. Suppose that
 $L=\bigoplus_{q\in Q} L_q=\bigoplus_{g\in G} L^{(g)}$ is a $G$-graded color Lie
 superalgebra. If   $L_0^{(e)}=L^{(e)}\cap L_0$ has finite dimension $m$, then  $L$ has a homogeneous soluble
 ideal of finite $(n,k,m)$-bounded codimension
 and of $(n,k)$-bounded derived length. \end{theorem}

Recall that by definition, all automorphisms of a color Lie
superalgebra $L=\bigoplus_{q\in Q} L_q$  preserve the given
$Q$-grading: $L_q^{\varphi}\subseteq L_q$ for all $q\in Q$.

 \begin{theorem}\label{th6}  Let  $Q$  be a finite cyclic group of  order  $k$. Suppose that a  color Lie
 superalgebra $L=\bigoplus_{q\in Q} L_q$  admits an automorphism $\varphi$ of finite
 order $n$ relatively prime to $k$. If the fixed-point subalgebra $C_{L_0}(\varphi)$ of $\varphi$ in $L_0$
 is
 finite-dimensional of dimension $m$, then $L$ has  a homogeneous soluble
 ideal of finite $(n,k,m)$-bounded codimension
 and of $(n,k)$-bounded derived length.\end{theorem}

In \cite{brgr} Bergen and Grzeszczuk  proved that if a color Lie
superalgebra  $L=\bigoplus_{q\in Q} L_q$, where $Q$ is a finite
abelian (not necessarily cyclic) group,  admits an automorphism of
finite order such that $C_L(\varphi)=0$, then $L$ is soluble. At
present, we do not know if this result can be extended to the case
of $\mathrm{dim}\,C_L(\varphi)=m$. The hypothesis is that $L$
contains a homogeneous soluble ideal of finite codimension with
bounds that do not depend on $|Q|$.

\vskip1ex It is clear that a Lie superalgebra is also a color Lie
superalgebra with $Q=Z_2$. In this particular case Theorems
\ref{th5} and \ref{th6} take the following forms.

\begin{corollary}\label{co1}
  Let  $G$ by a finite cyclic group of odd order $n$ and let
 $L=L_0\oplus L_1=\bigoplus_{g\in G} (L_0^{(g)}\oplus L_1^{(g)})$ be a  $G$-graded Lie
 superalgebra over an arbitrary field  $F$,  that is  $[L^{(g)},\, L^{(h)}]\subseteq
L^{(gh)}$ and  $L^{(g)}=\big(L^{(g)}\cap L_0\big)\oplus
\big(L^{(g)}\cap L_1\big)$. If $L_0^{(e)}= L^{(e)}\cap L_0$ has
finite dimension $m$, then $L$ has a  homogeneous soluble  ideal
of finite $(n,m)$-bounded codimension  and of $n$-bounded derived
length.
\end{corollary}

\begin{corollary}\label{co2}   If a  Lie
 superalgebra $L=L_0\oplus L_1$  admits an automorphism $\varphi$ of finite
 odd order $n$ such that the fixed-point subalgebra $C_{L_0}(\varphi)$ of $\varphi$ in
$L_0$
 is
 finite-dimensional of dimension $m$, then $L$ has  a homogeneous soluble
 ideal of finite $(n,m)$-bounded codimension
 and of $n$-bounded derived length.
\end{corollary}

\section{Preliminaries}

 We will use the square brackets $[\,,\,]$ for the multiplicative operation.  If $M,\, N$ are subspaces of an algebra $L$
then $[M,N]$ denotes the subspace, generated  by all the products
$[m,n]$ for $m\in M$, $n\in N$. If $M$ and $N$ are two-side
ideals, then $[M,N]$ is also a two-side ideal; if $H$ is a
(sub)algebra, then $[H,H]$ is its two-side ideal and, in
particular, its subalgebra.
 The  subalgebra  generated by  subspaces~$U_1,U_2,\ldots, U_k$ is
denoted by $\left<U_1,U_2,\ldots, U_k\right>$, and the two-side
ideal generated by~$U$ is denoted by ${}_{\rm
id}\!\left<U_1,U_2,\ldots, U_k\right>$.

\vskip1ex
 A simple product  $[a_1,a_2,a_3,\ldots, a_s]$ is by
definition the left-normalized product
$[...[[a_1,a_2],a_3],\ldots, a_s]$. The analogous notation is also
used for subspaces
$$[A_1,A_2,A_3,\ldots, A_s]=[...[[A_1,A_2],A_3],\ldots, A_s].$$

\vskip1ex
 The derived series of an algebra $L$ is defined as
$$L^{(0)}=L, \; \; \; \, \, \, \, \, \, L^{(i+1)}=[L^{(i)},L^{(i)}].$$

\vskip1ex If   $L=\bigoplus_{i=0}^{n-1}L_i $ is a  $({\Bbb Z}
/n{\Bbb Z} )$-graded algebra, elements of the $L_a$ are called
\textit{homogeneous} (with respect to this grading), and products
in homogeneous elements \textit{homogeneous products}. A subspace
 $H$  of $L$ is said to be \textit{homogeneous}
if $H=\bigoplus_{i=0}^{n-1} (H\cap L_i)$; then we set $H_i=H\cap
L_i$. Obviously, any subalgebra or an ideal generated by
homogeneous subspaces is
 homogeneous. A homogeneous subalgebra  can be regarded as a
$({\Bbb Z} /n{\Bbb Z} )$-graded algebra with the induced grading.
It follows that the terms of the derived series of $L$, the ideals
$L^{(k)}$, are also $({\Bbb Z} /n{\Bbb Z} )$-graded algebras with
induced grading $L^{(k)}_i=L^{(k)}\cap L_i$, and
$$L^{(k+1)}_i=\sum_{u+v\equiv i\,({\rm mod\,}n)
}[L^{(k)}_{u},\,L^{(k)}_{v}].$$

\vskip1ex
 By  property  (\ref{tozh}) if $L$ is a  $({\Bbb Z}
/n{\Bbb Z} )$-graded Lie type algebra over a field $\Bbb F$, then
for all homogeneous $a,b,c\in L$ there exist $0\neq\alpha\in \Bbb
F $ $\beta\in \Bbb F$ such that
$$[a,[b,c]]=\frac{1}{\alpha}\,[[a,b],c]-\frac{\beta}{\alpha}\,[[a,c],b].$$ Hence
 any (complex) product in certain homogeneous elements in $L$  can be expressed as a linear combination of simple products of
the same length in the same elements.  It follows that the
(two-side) ideal in  $L$ generated by a homogeneous subspace $S$
is the  subspace generated by all the homogeneous simple products
$[x_{i_1},y_j,x_{i_2},\ldots, x_{i_t}]$ and $[y_j,x_{i_1},
x_{i_2},\ldots x_{i_t}]$, where $t\in \Bbb N$ and $x_{i_k}\in L,
y_j\in S$ are homogeneous elements. In particular, if $L$ is
generated by a homogeneous subspace $M$, then  its space is
generated by simple homogeneous products in elements of~$M$.

\section{Solubility criterion}

In this section we will use the next shortened notation:
$$[\beta^{\,s},\,\alpha,\,\beta^{\,r}]=[\underbrace{\beta,\ldots, \beta}_{s},\alpha,\underbrace{\beta,\ldots
\beta}_{r}] =[...[[[\underbrace{\beta,\ldots,
\beta}_{s}],\alpha],\underbrace{\beta],\ldots, \beta}_{r}]$$ where
$\alpha$ and $\beta$ are subspaces  of an algebra~$L$. In
particular, $[\beta^{\,s},\beta^{\,r}]=[\beta^{\,s+r}]\neq
[\beta^{\,s},\,[\beta^{\,r}]]$.

\begin{proposition}\label{p1} There exists a function  $f:\Bbb
N\times \Bbb N \rightarrow \Bbb N$ such that for any   $(\Bbb Z
/n\Bbb Z )$-graded Lie type algebra $L$  its
 $f(m,n)$th term of the derived series  $L^{(f(m,n))}$ is contained in the subalgebra generated by the
 subspace
$$\sum_{t=0}^{m}[L_0^t,L,L_0^{m-t}],$$ where $L_0$ is the
zero component.\end{proposition}

\begin{proof} For the convenience we introduce the following notation:
$$S_k(r)=\sum_{t=0}^{r}[L_0^t, L_k, L_0^{r-t}].$$ In the next
auxiliary lemma we establish some elementary properties of the
subspaces
 $S_k(r)$.

\begin{lemma}\label{l1} The following inclusions hold:

\vskip1ex {\rm (a)} $[S_k(r),\,L_0]+ [L_0,\,S_k(r)]\subseteq
S_k(r+1);$

\vskip1ex  {\rm (b)} $S_k(r+1)\subseteq S_k(r);$

\vskip1ex  {\rm (c)} $[\left< S_1(r),\ldots, S_{k-1}(
r),L_{k+1},\ldots, L_{n-1}\right>,\, L_0] +[
L_0,\,\left<S_1(r),\ldots, S_{k-1}( r),L_{k+1},\ldots,
L_{n-1}\right>]\subseteq \left<S_1(r),\ldots, S_{k-1}(
r),L_{k+1},\ldots, L_{n-1}\right>.$
\end{lemma}

\begin{proof}

(a) By definition  $$[S_k(
r),\,L_0]=\sum_{t=0}^{r}[L_0^t,\,L_k,\,L_0^{r-t},\,L_0]\subseteq
S_k (r+1).$$ By (1) for all homogeneous elements $a,b,c$ there
exist $0\neq\alpha, \beta \in \Bbb F$ such that
$$[a,[b,c]]=\frac{1}{\alpha}[a,b,c]-\frac{\beta}{\alpha}[a,c,b].$$ It
follows that a
 product $[a, [b_1,b_2,\ldots b_s]]$ in homogeneous elements can be expressed as a linear combination of simple
products  $[a, b_{i_1},\ldots, b_{i_s}]$ of the same length in the
same elements.  Hence
 $$[L_0,\,S_k(r)]=[L_0,\,\sum_{t=0}^{r}[L_0^t,\,L_k,\,L_0^{r-t}]]=\sum_{t=0}^{r}[L_0,[L_0^t,\,L_k,\,L_0^{r-t}]]\subseteq
\sum_{t=0}^{r}[L_0^{t+1},\,L_k,\,L_0^{r-t}]\subseteq S_k (r+1)$$
and thus (a) holds.

 \vskip1ex (b) Since $[L_0,L_k]\leq L_k$, $[L_k,L_0]\leq L_k$ and
$[L_0^2]\leq L_0,$  we have
$$S_k(
r+1)=\sum_{t=0}^{r+1}[L_0^t,L_k,L_0^{r+1-t}]=[L_k,L_0^{r+1}]+[L_0,L_k,L_0^r]+
\sum_{t=0}^{r-1}[L_0^{2+t},L_k,L_0^{r-1-t}]\subseteq$$
$$\subseteq [L_k,L_0^{r}]+[L_k,L_0^{r}]+ \sum_{t=0}^{r-1}
[L_0^{t+1},L_k,L_0^{r-1-t}] \subseteq [L_k,L_0^{r}]+[L_k,L_0^{r}]+
\sum_{t=1}^{r} [L_0^{t},L_k,L_0^{r-t}]\subseteq    S_k(r).$$

\vskip1ex (c) An element of the subalgebra $$\left< S_1(r),\ldots,
S_{k-1}( r),L_{k+1},\ldots, L_{n-1}\right>$$ is a linear
combination of simple products in  elements from $S_i( r)$,
$i=1,2,\ldots, k-1$ and  $L_i$, $i=k+1,\ldots, n-1$ of the form
$$[a_1,a_2,\ldots, a_q],
$$
where each $a_i$ is an   element of $S_i( r)$ or $L_i$.  Let
$l_0\in L_0$. The product $\big[l_0,[a_1,a_2,\ldots, a_q]\big]$
from $[L_0,\,\left< S_1(r),\ldots, S_{k-1}( r),L_{k+1},\ldots,
L_{n-1}\right>]$ can be represented as a linear combination of
products of the form
$$[l_0,a_{i_1},\ldots, a_{i_q}],
$$
 where  $a_{i_j}\in\{a_1,\ldots a_q\}$. In view of assertion (a)
the element $[l_0,a_{i_1}]$ belongs to the same subspace
($S_{i_1}(r)$ or $L_{i_1}$) as  $a_{i_1}$. In both cases the
product  $[l_0,a_{i_1},\ldots, a_{i_q}]$ belongs to the subalgebra
$$\left< S_1(r),\ldots, S_{k-1}( r),L_{k+1},\ldots,
L_{n-1}\right>.$$

Consider now  a product $$\big[[a_1,a_2,\ldots a_q],l_0\big]$$
from
$$[\left<S_1(r),\ldots, S_{k-1}( r),L_{k+1},\ldots,
L_{n-1}\right>,\, L_0].$$  By  (\ref{tozh}), we transfer the
element $l_0$ to the left aiming to obtain a linear combination of
elements of the form
$$[a_1,a_2,\ldots, [a_j,l_0],\ldots, a_q].
$$
In view of  assertion (a), each  element $[a_j,l_0]$  belongs to
the same subspace  ($S_j(r)$ or $L_j$) as  $a_j$. Hence, the
products  $[a_1,a_2,\ldots, [a_i,l_0],\ldots, a_q]$ are contained
in the subalgebra
 $$\left< S_1( r),\ldots, S_{k-1}( r),L_{k+1},\ldots,
L_{n-1}\right>$$ as well.
\end{proof}

You can find the  proof of the next elementary lemma in~\cite{kr}
(see, also~\cite[Lemma 4.3.5]{kh-book}).

\begin{lemma}\label{l2} If $i+j\equiv  k\, {\rm (mod}\, n)$ for $0\leq i,\, j \leq
n-1$, then the numbers  $i$ and $j$ are both greater than $k$ or
less than $k$.\end{lemma}

We now prove Proposition~\ref{p1}. We establish that for some
functions  $h_i :\Bbb N\times \Bbb N \rightarrow \Bbb N$, \
$i=1,\,2$, and  for $k=1,\,2,\,\ldots ,n$ the following inclusions
hold:

\begin{equation}\label{1usl}
\hskip2em L^{(h_1(m,k))}\cap L_k\,\subseteq
\left<S_1(m2^{n-k}),\ldots ,S_{k-1}(m2^{n-k}),\, L_{k+1}, \ldots ,
L_{n-1} \right>\,\,+\,\,S_k( m2^{n-k}),
\end{equation}
\begin{equation}\label{2usl}
L^{(h_2(m,k))}\,\subseteq \left<S_1( m2^{n-k}),\ldots ,
S_k(m2^{n-k}),\, L_{k+1}, \ldots, L_n\right>. \end{equation} We
extend  the statement~(\ref{2usl}) to the case $k=0$ and consider
the equality $L=\bigoplus _{i=0}^{n-1}L_{i}$ as the base of
induction for~(\ref{2usl}) with $h_2(m,0)=0$. At each step for a
given $k$ we first prove~(\ref{1usl}) by using  the induction
hypothesis for~(\ref{2usl}). Then  the statement~(\ref{2usl}) is
deduced from~(\ref{1usl}) for $k$ and the induction hypothesis
for~(\ref{2usl}).

\vskip1ex In order to establish (\ref{1usl}), we  prove the
following chain of inclusions:
\begin{equation} \label{3usl}L^{(r(h_2(m,k-1)+1))}\cap L_k \,
\,\subseteq \,\,\left<S_1( m2^{n-k}),\ldots ,S_{k-1}(m2^{n-k}), \,
L_{k+1}, \ldots , L_{n-1}\right> \,\,+\,\,S_k(r),\end{equation}
where $r=1,2,\ldots,m2^{n-k}$. The statement (\ref{1usl}) will
follow from~(\ref{3usl}).

\vskip1ex Let  $r=1$. If  $a\in L^{(h_2(m,k-1)+1)}\cap L_k$, then
$a$ is equal to a linear combination of products of the form
$[b,c]$, where $b,c\in L^{(h_2(m,k-1))}$ and $b,c$ are
homogeneous. By the induction hypothesis  the
inclusion~(\ref{2usl}) holds for  $k-1$; hence, the elements $b$
and $c$, and therefore  $[b,c]$, are contained in the subalgebra
$$\left<S_1(m2^{n-k+1}),\ldots ,S_{k-1}(m2^{n-k+1}),
L_k, \ldots , L_n\right>.$$ Then $[b,c]$ can be expressed as a
linear combination of simple products in homogeneous elements of
the subspaces indicated inside the angle brackets.  Every such
simple product has the form  $[u,v]$, where $u$ is its initial
segment and $v$ is the last element, which is contained in one of
the indicated subspaces.   If
 $v\in L_q$ for $q\in \{ k, \, k+1, \ldots , n-1, \, n\} $, then
$u\in L_s$ for $ s+q\equiv k\, ({\rm mod \,}n)$, where $s\in \{
0,\,1,\ldots n-1\}$. If $k<q<n$, then $k<s<n$ by Lemma~\ref{l2}
and, consequently,  $[u,v]$ is contained in  $ \left< L_{k+1},
L_{k+2}, \ldots , L_{n-1} \right> $ and  therefore in the right
side of~(\ref{3usl}). If $q=k$ or $q=n$, then, respectively, $s=0$
or $s=k$; in both cases  $[u,v]$ lies in the subspace
$[L_k,L_0]+[L_0,L_k]=S_k(1)$, which is also contained in the right
side of~(\ref{3usl}) in the case $r=1$ under consideration. Let
$v\in S_q (m2^{n-k+1})$
 for $q\in \{ 1,\,2,\ldots , k-1\} $. Then
 $$[u,v]\in
 [L_s,\,S_q(m2^{n-k+1})]\subseteq \sum_{i+j=m2^{n-k+1}}
 \big[L_s,\,[L_0^{i},\,L_q,\,L_0^{j}]\big]\subseteq \sum_{i+j=m2^{n-k+1}}
 [L_s,L_0^{i},\,L_q,\,L_0^{j}].
 $$
If $j\geq 1$,  the products $[L_s,L_0^{i},L_q,L_0^{j}]$ are
obviously contained in  $[L_k,L_0]\subseteq S_k(1)$. If $j=0$, in
the summand $[L_s,L_0^{m2^{n-k+1}},L_q]$ we move $L_q$ to the left
by~$(\ref{tozh})$.  At the first step, say, we get
$$[L_s,L_0^{m2^{n-k+1}},L_q]\subseteq
[L_s,L_0^{m2^{n-k}},\,L_q,\,L_0]+\big[L_s,\,L_0^{m2^{n-k}},\,[L_0,L_q]\big].$$
The first summand  lies in $[L_k,L_0]$, which, in turn, is
contained in the second summand of the right part of~(\ref{3usl}).
In the second summand  the subspace $[L_0,L_q]$ takes over the
role of $L_q$ and is also moved to the left, over the $L_0$. As a
result we obtain a sum of  products that are  contained in
$[L_k,L_0]$ and the summand
$[L_s,L_0^{m2^{n-k}},\underbrace{[L_0[L_0[\ldots[L_0}_{m2^{n-k}}L_q]...]]]].$
We assert that the last summand  lies in the first summand of the
right side of~(\ref{3usl}). In fact,
$$[L_s,L_0^{m2^{n-k}}]\subseteq
\sum_{t=0}^{m2^{n-k}}[L_0^t,L_s,L_0^{m2^{n-k}-t}]=S_s(m2^{n-k})$$
and
$$\underbrace{[L_0[L_0[\ldots[L_0}_{m2^{n-k}}L_q]...]]]\subseteq
\sum_{t=0}^{m2^{n-k}}[L_0^t,L_q,L_0^{m2^{n-k}-t}]=S_q(m2^{n-k}),$$
where $1\leq s,\,q\leq k-1$.

\vskip1ex
  For $r>1$ we apply the established statement~(\ref{3usl}) for $r=1$ to the algebra
$L^{((r-1)(h_2(m,k-1)+1)}$ with induced grading instead of~$L$:
$$L^{(r(h_2(m,k-1)+1))}\cap
L_k\,=\,\left(L^{((r-1)(h_2(m,k-1)+1))}\right)^{(h_2(m,k-1)+1)}\,\cap
\,L_k\,\,\subseteq$$
$$\subseteq \left<S_1(m2^{n-k}),\,
 \ldots ,
S_{k-1}(m2^{n-k}),\, L_{k+1}, \ldots , L_{n-1}\right>+$$
$$\;\;\;\;\;\;\;\;\,\,+\,\,
[L^{((r-1)(h_2(m,k-1)+1)} \cap \,L_k,\,\,
L_0]\,\,+\,\,[L_0,\,\,L^{((r-1)(h_2(m,k-1)+1))} \cap \,L_k].
$$
By
using the obvious inclusions, we enlarged the first summand and
get the same as in~(\ref{3usl}). We now apply the induction
hypothesis for $r-1$ to the second and third summands:
$$[L^{((r-1)(h_2(m,k-1)+1))} \cap \,L_k,\,\,L_0]\,\,\subseteq
$$
\begin{equation}\label{v1} \subseteq
\big[\left< S_1(m2^{n-k}),\,
 \ldots ,S_{k-1}(m2^{n-k}),\, L_{k+1}, \ldots ,
 L_{n-1}\right>,\,\, L_0\big]
\,+\,[S_k(r-1),\, L_0] \end{equation}
$$
[L_0,\,\,L^{((r-1)(h_2(m,k-1)+1))} \cap \,L_k]\,\,\subseteq
$$
\begin{equation}\subseteq \label{v2} \big[L_0, \,\, \left< S_1(m2^{n-k}),\,
 \ldots ,S_{k-1}(m2^{n-k}),\, L_{k+1}, \ldots,
 L_{n-1}\right>\big]
\,+\,[L_0,\,S_k(r-1)]. \end{equation} In view of  (a) and (b) of
Lemma~\ref{l1}  the right parts of~(\ref{v1}) and~(\ref{v2}) are
contained in the right part of~(\ref{3usl}). This completes the
proof of  the inclusions~(\ref{3usl}) for all~$r$.

\vskip1ex For  $r=m2^{n-k}$ and  $h_1(m,k)=m2^{n-k}(h_2(m,k-1)+1)$
the inclusion~(\ref{3usl}) is exactly the inclusion~(\ref{1usl})
for~$k$.

\vskip1ex We now put   $h_2(m,k)=h_2(m,k-1)+h_1(m,k)$ and prove
the assertion~(\ref{2usl}) for this value of the function
$h_2(m,k)$. We have $L^{(h_2(m,k))}=\left( L^{(h_1(m,k))}
\right)^{(h_2(m,k-1))}$. We apply  statement~(\ref{2usl}) for
$k-1$ to the subalgebra $L^{(h_1(m,k))}$ with the inducing
grading:
$$
\left( L^{(h_1(m,k))} \right)^{(h_2(m,k-1))}\subseteq
$$
\begin{equation}\label{v3}\subseteq \left< S_1(m2^{n-k}),\,
 \ldots , S_{k-1}(m2^{n-k}),
         \,\,L^{(h_1(m,k))}\cap  L_{k},\,\, L_{k+1},
\ldots , L_{n-1}, L_n\right>. \end{equation} Here we have used the
inclusions $ S_i(m2^{n-k+1})\subseteq S_i(m2^{n-k})$ for $
i=1,\,2,\ldots ,k-1$ and  $ L^{(h_1(m,k))}\cap L_{j}\subseteq L_j$
for $j=k+1,\ldots ,n-1,\,n$. Substituting the established
inclusion (\ref{1usl}) for $L^{(h_1(m,k))}\cap  L_{k}$ in
(\ref{v3}) after the removal of the  repetitions  we obtain
$$L^{(h_2(m,k))}=\left( L^{(h_1(m,k))} \right)^{(h_2(m,k-1))}\,\,\subseteq $$
$$\subseteq\,\,\big\langle S_1(m2^{n-k}),\,
 \ldots , S_{k-1}(m2^{n-k}),
\, S_k(m2^{n-k}) ,\, L_{k+1},  \ldots , L_{n-1}, L_n\big\rangle
.$$ This is the required inclusion (\ref{2usl}) for $k$.

\vskip1ex For $k=n$ the inclusion~(\ref{2usl}) takes the form
$$
L^{(h_2(m,n))}\,\,\subseteq \,\,\big\langle S_1(m),\, S_2(m),
\ldots , S_{n-1}(m),\, S_{n}(m)\big\rangle,$$ which is the
statement of Proposition~\ref{p1} with $f(m,n)=h_2(m,n)$.
\end{proof}

\section{Representatives and Generalized centralizers}\label{s-rep-cen}

In this section we construct the generalized centralizers which
are certain subspaces of the homogeneous components  $L_i$, $i\neq
0$:
$$L_i=L_i(0)\geq L_i(1)\geq \cdots \geq L_i(N(n)).$$
Constructing the generalized centralizers is carried out by
induction on the level, which is a parameter taking integer values
from $0$ to some $n$-bounded number  $N=N(n)$. Simultaneously with
the construction of these subspaces certain homogeneous elements,
called representatives, are being fixed.

\vskip1ex {\bf Index Convention.} {\it In what follows an element
of the homogeneous compo\-nent~$L_i$ will be denoted by a small
letter with index $i$
 and the index will only indicate
the homogeneous compo\-nent where this element belongs: $x_i \in
L_i$. To lighten the notation we will not be using numbering
indices for elements of the $L_j$, so that different
 elements can be denoted by the same symbol when
 it only matters which homogeneous compo\-nents these
elements belong to. For example, $x_3$ and $x_3$ can be different
elements of $L_3$. These indices will be
 regarded as residues modulo~$n$; for example, $a_{-i}\in
L_{-i}=L_{n-i}$.}

\vskip1ex  The {\it pattern\/} of a product in homogeneous
elements (of~$L_i$) is its bracket structure together with the
arrangement of the indices under the Index Convention. The {\it
length\/} of a pattern is the length of the product. The product
is said to be the {\it value of its pattern} on the given
elements. For example, $[a_1,[b_2,b_2]]$ and $[x_1,[z_2,y_2]]$ are
values of the same pattern of length~3. Note that under the Index
Convention the elements $b_2$ in the first product can be
different.

\vskip1ex Let $j\ne 0$. For every ordered tuple  of elements $\vec
x=(x_{i_1},\dots ,x_{i_k}),$
 $x_{i_s}\in L_{i_s}$,  $i_s\ne 0$, such that
$j+i_1+\dots + i_k \equiv 0 \,(\mbox{mod}\, n)$ we define the
mappings
 $\vartheta _{t,\vec x}: L_j \rightarrow L_0$,
$t=0,\ldots, k$:
$$\vartheta _{t,\vec x}:\, \, y_j\rightarrow
[x_{i_1},\,x_{i_2},\,\ldots,\, x_{i_t},\, y_j,
\,x_{i_{t+1}},\ldots,\, x_{i_k}].$$ By linearity they   are
homomorphisms of the subspace
 $L_j$ into
$L_0$. Since  $\dim L_0\leq m$, we have ${\rm dim}\,
(L_j/\mbox{Ker}\, \vartheta _{t,\vec x})\leq m$ for all $\vec x$,
$t$.

 \vskip1ex
{\bf Definition of level 0.} At level 0 we only fix
representatives of level $0$. For each  pattern ${\bold
P}=[\ast_i,\,\,\ast_{-i}]$ of a product of length $2$ with
non-zero indices $ \pm i\neq 0$, among   all values
 of ${\bold P}$ on  homogeneous elements of $L_{i},$ $i\ne 0$, we
choose elements $c \in L_0$ that form a basis of the subspace
spanned by all values of ${\bold P}$ on  homogeneous elements of
 $L_{i}$, $i\ne 0$. The same is done for every pattern
 ${\bold P}=[\underbrace{\ast_i,\,\ldots, \ast_i}_n]$ of a simple product of length $n$ with one and the
same index $i\ne 0$ repeated $n$ times.  The elements of
$L_{j}$,\, $j\neq 0$, involved in these fixed representations of
the products $c$ are called
 {\it representatives of level $0$} and denoted by  $x_j(0)\in L_j$
(under the Index Convention). Since the total number of patterns
${\bold P}$
 under consideration is $n$-bounded and the dimension of the subspace  $L_0$ is at
 most  $m$, the number of representatives of level $0$ is
$(m,n)$-bounded.

\vskip1ex

 Before we describe the induction step we choose an
increasing sequence of positive integers $W_1<W_2<\ldots <W_N$,
all of which are  $n$-bounded but sufficiently large compared to
$n$-bounded values of some other parameters of the proof.
Moreover, the differences   $W_{k+1}-W_k$ must be also
sufficiently large in the same sense (see  \S\, \ref{parametry}
for the exact values of theses parameters).

\vskip1ex

 {\bf Definition of level
$s>0$.} Unlike the level $0$, representatives of level  $s>0$ are
defined in two different ways and are accordingly called either
{\it $b$-representatives} or {\it $x$-representatives.} Suppose
that we have already fixed
 $(m,n)$-boundedly  many representatives of level $<s$, which are
 either $x$-representatives of the form
  $ x_{i_k}(\varepsilon_k)\in L_{i_k}(\varepsilon_k)$ or  $b$-representatives of the form
 $ b_{i_k}(\varepsilon_k)\in L_{i_k}$, $i_k\ne 0$, of levels $\varepsilon_k< s$.

\vskip1ex
 We define the  {\it generalized centralizers of level  $s$} (or, for short,
centralizers of level $s$), by setting for each non-zero $j$
$$L_j(s)=\bigcap_{\vec z}\,\bigcap_{t}\mbox{Ker}\, \vartheta  _{t,\vec z},$$
where  $\vec z=\left( z_{i_1}(\varepsilon_1),\, \ldots ,\,
z_{i_k}(\varepsilon_k)\right) $ runs over all possible ordered
tuples of all lengths
 $k\leq W_s$ consisting of representatives of (possibly different)
levels $<s$ (i.~e. $z_{i_u}(\varepsilon_u)$ denote elements of the
form
 $x_{i_u}(\varepsilon_u) $  or $b_{i_u}(\varepsilon_u)$, $\varepsilon_u<s$,
in any combination) such that
$$ j+ i_1+\cdots + i_k\equiv 0\, (\mbox{mod}\, n),$$ and $t=0,\ldots, k$.
 The elements of  $L_j(s)$ are also called  {\it
centralizers of level $ s$} and denoted by
 $y_j(s)$ (under the Index Convention).

 The number of representatives of all levels $<s$ is
 \ $(m,n)\hs$-bounded, the tuples $\vec z$ have  $n\hs$-bounded length,
 and  ${\rm dim}\,L_j/\mbox{Ker}\, \vartheta _{t,\vec
z}\leq m$ for all $\vec z$, $t$.  Hence the intersection here is
taken over an   $ (m,n) \hs$-bounded number of subspaces of
$m\hs$-bounded codimension in   $L_j$, and therefore
 $L_j(s)$
 also has  $(m,n)\hs$-bounded codimension in the subspace $L_{j}$.

\vskip1ex Now we  fix representatives of level $s$.  First, for
each nonzero
 $j$ we fix an arbitrary basis of the factor-space $L_j/L_j(s)$ and
for each element of the basis we choose arbitrarily  a
representative in $L_j$. These elements are denoted by
 $ b_j(s)\in L_j$ (under the Index Convention) and  are called
{\it $b$-representatives of level $s$}. The total number of
 $b$-representatives of level $s$  is
 $(m,n)\hs$-bounded, since the dimensions of $L_j/L_j(s)$  is
$(m,n)\hs$-bounded for all $j\neq 0$.

Second, for each pattern ${\bf P}=[\ast_i,\,\,\ast_{-i}]$ of
length $2$ with non-zero indices $ \pm i\ne 0$ among   all values
 of this pattern on  homogeneous elements of $L_{i}(s),\, i\ne 0$, we
choose products that form a basis of the subspace spanned by all
values of of this pattern on homogeneous elements of
$L_{i}(s),\,i\ne 0$. The elements involved in these products are
called {\it $x$-representatives of level $s$} and are denoted by $
x_{j}(s)$ (under the Index Condition). Since the number of
patterns under consideration is  $n$-bounded and  the dimension of
the subspace
 $L_0$ is at most $m$, the total number of $x$-representatives of level  $s$
is $(m,n)\hs$-bounded. Together elements of the form  $b_i(s)$ and
$x_j(s)$ are sometimes called simply   {\it representatives of
level $s$}. Note that  $x$-representatives of level $s$, elements
$x_j(s)$,  are also centralizers of level $s$, but
$b$-representatives, elements $b_i(s)$, are not.

\vskip1ex It is clear from the construction that
\begin{equation}\label{basic-inclusion}
L_j(k+1)\leq L_j(k) \end{equation} for all $j\neq 0$ and any $k$.

\vskip1ex

By definition a centralizer  $y_v(s)$ of any level
 $s$ has the following centralizer property with respect to representatives
 of lower levels:
\begin{equation}\label{basic-property}
[z_{i_1}(\varepsilon_1),\,\ldots,\,
z_{i_{t}}(\varepsilon_{t}),\,y_v(s),\,z_{i_{t+1}}(\varepsilon_{t+1}),\,\ldots,\,
 z_{i_k}(\varepsilon_k)]=0, \end{equation}
whenever  $ v+ i_1+\cdots +i_k\equiv 0\, (\mbox{mod}\, n)$, \
$k\leq W_s$,  $t\in \{0,\ldots, k\}$ and the elements
$z_{i_j}(\varepsilon_j)$ are representatives (i.~e. either
$b_{i_j}(\varepsilon_j)$
 or $x_{i_j}(\varepsilon_j)$, in any combination) of any  (possible different) levels $\varepsilon_j<s$.

\vskip1ex The next lemma permits to represent products from $L_0$
as linear combinations of products in representatives; we shall
refer to this lemma as the ``freezing'' procedure.

\begin{lemma} [Freezing procedure]\label{l6}  Each  product of the form
$[a_{-j},\,b_j]\in L_0$, where $j\ne 0$, and each simple product
of length $n$ in homogeneous elements  with one and the same index
 $i\ne 0$, repeated  $n$ times  can be represented  $($frozen\/$)$ as a linear
combination of products of the same pattern in representatives of
level  $0$.

Each product $[y_{-j}(k),\,y_j(l)]\in L_0$
 in centralizers of levels $k,l$ can be represented $($frozen\/$)$ as a linear combination of products
 $[x_{-j}(s),\,x_j(s)]$ of the same pattern in $x$-representatives of any level $s$ satisfying $ 0\leq s\leq \min \{ k,l\}$.
 \end{lemma}

\begin{proof}
The lemma follows directly from the definitions of level~0 and
levels~$s>0$ and from the inclusions (\ref{basic-inclusion}).
\end{proof}

\vskip1ex {\it An $x$-quasirepresentative of length  $w$ and
level $k$\/} is any product of length  $w\geq 1$ involving exactly
one $x$-representative  $x_i(k)$ of level $k$ and $w-1$
representatives of lower levels, elements of the form
$b_{i_k}(\varepsilon_k)$ or
 $x_{i_j}(\varepsilon_j)$, in any combination and of any levels $\varepsilon_s<k$.
 $x$-Quasirepresentatives of level $k$ (and only they) are    denoted by
$\hat{x}_{j}(k)\in L_j$ under the Index Convention, where,
clearly, $j$ is equal  modulo $n$ to the sum of the indices of all
the elements involved in the $x$-quasirepresentative.
$x$-Quasirepresentatives of length $1$ are precisely
$x$-representatives.

\vskip1ex A {\it quasirepresentative of length $w$ of level  $\leq
k$\/} is any product of length $w$ in representatives of level
$\leq k$, elements of the form either  $b_{i_k}(\varepsilon_k)$ or
$x_{i_j}(\varepsilon_j)$, in any combination and of any levels
$\varepsilon_s\leq k$.
 Quasirepresentatives of level $k$ are exclusively denoted by
$\hat{b}_{j}(k)\in L_j$ under the Index Convention, where  $j$ is
equal modulo  $n$ to the sum of the indices of all elements
involved in the quasirepresentative. It is clear that a product in
quasirepresentatives is also a quasirepresentative of length equal
to the sum of the lengths of the quasirepresentatives involved and
of level equal to the maximum of their levels.

\vskip1ex A {\it quasicentralizer of length $w$ of level
 $k$\/} is any product involving exactly one centralizer $y_i(k)\in L_i(k)$ of
 level $k$ and  $w-1$ representatives of lower levels, elements of the
 form
$b_{i_k}(\varepsilon_k)$ or $x_{i_j}(\varepsilon_j)$, in any
combination and of any levels  $\varepsilon_s<k$.
 Quasicentralizers of level $k$ are exclusively denoted by
$\hat{y}_{j}(k)\in L_j$ under the Index Convention; the index $j$
is equal modulo  $n$ to the sum of the indices of all the elements
involved.

It is clear that an $x$-quasirepresentative of level $k$ is also a
quasicentralizer of level $k$; this does not  apply to all
quasirepresentatives.

\begin{lemma}[{\cite[Lemma 2]{ma-khu}}]\label{lkvasic} Any product
involving exactly one quasicentralizer $\hat{y}_{i}(t)$ of level
$t$ and quasirepresentatives of levels $< t$ is equal to
 $0$ if the sum of the indices of all elements involved is equal to $0$
and the sum of their lengths  is at most
 $W_t+1$. \end{lemma}

 \begin{proof} Applying  (\ref{tozh}), we represent the product
as a linear combination of simple products of length
 $\leq W_t+1$ involving only one  centralizer of level $t$ and some representatives of levels $<t$.
Since the sum of the indices of all these elements is also equal
to $0$, all these products are equal to $0$
by~(\ref{basic-property}).
 \end{proof}

 \begin{lemma}[{\cite[Lemma 5]{ma-khu}}]\label{l-kv-cen} Any quasicentralizer
$\hat{y}_{j}(l+1)$
 of level $l+1$ and of length at most $W_{l+1}-W_l+1$ is a
 centralizer of level
$l$, i.~e. $\hat{y}_{j}(l+1)\in L_j(l)$.

\end{lemma}

\begin{proof} The element $\hat{y}_{j}(l+1)$ is a linear combination of simple
products involving only one centralizer of level $l+1$, the
element  $y_t(l+1)\in L_t(l+1)$ for some  $l$, and at most
$W_{l+1}-W_l$ representatives of lower levels $ \leq l$.
Substituting this expression into
$$[z_{i_1}(\e _1),\, \ldots, \,\hat{y}_{j}(l+1),\,\ldots,
 z_{i_k}(\e _k)] \eqno{(13)} $$
where  $k\leq W_{l}$, \
 $z_{i_k}(\e _k)$ are representatives of levels $\e _s<l$, and $ j+ i_1+\cdots +i_k\equiv 0\, (\mbox{mod}\,
n)$ we obtain a linear combination of simple products of length at
most $1\,+\,W_{l+1}-W_l+W_l=1\,+\,W_{l+1}$. The sum of the indices
remains equal $0$. Hence each summand is equal to $0$ by
(\ref{basic-property}).
\end{proof}

 \begin{lemma}[{\cite[Lemma 3]{ma-khu}}]\label{l-product} A product  of the form
$[a_{-i},\,y_i(k)]$ \/ $($or $[y_i(k),\,a_{-i}]$\/ $)$, where
$y_i(k)$ is a centralizer of level
 $k>1$, is equal to a  product of the form
 $[y_{-i}(k-1),\,y_i(k)]$ \/ $($or $[y_i(k),\,y_{-i}(k-1)]$ respectively\/ $)$,  where
$y_{-i}(k-1)$ is a centralizer of level $k-1$.
\end{lemma}

\begin{proof} We represent $a_{-i}$ as a sum of a linear combination of
elements of the form  $b_{-i}(k-1)$ for some $ b$-representatives
and a centralizer $y_{-i}(k-1)$ of level $k-1$.
 Then the product
$[a_{-i},\,y_i(k)]$ can be represented as a sum of a linear
combination of elements  of the form $[b_{-i}(k-1),\,y_i(k)]$ and
the product  $[y_{-i}(k-1),\,y_i(k)]$.  Since
$[b_{-i}(k-1),\,y_i(k)]=0$ by~(\ref{basic-property}) we get
$[a_{-i},\,y_i(k)]=[y_{-i}(k-1),\,y_i(k)]$.
 Similarly, $[y_i(k),\,a_{-i}]=[y_{i}(k),\,y_{-i}(k-1)],$
 where $y_{-i}(k-1)$ is a centralizer of level $k-1$.
\end{proof}

\vskip1ex {\bf Notation.} Because of the special role of the
number $n=|\varphi |$, the greatest common divisor  $(n,k)$ of
integers $n$ and $k$ will be denoted by $\overline{ k}$ for short.
Clearly, $\overline{n+k}=\overline{k}$ and $\overline{
(k,l)}=(\overline{k},\overline{l})$ is the greatest common divisor
of three integers $n$, $k$ and $l$.

\begin{lemma}[{see  \cite[Lemma 4]{ma-khu}}]\label{l2n} Any simple product of length $2n$ of
the form
\begin{equation}\label{2n} [a_s,
\,\hat{y}_{j}(n_1),\,\hat{y}_{j}(n_2),\,\ldots,
\,\hat{y}_{j}(n_{2n-1})]\end{equation}
 is equal to
 $0$ if  $\overline{ j}$ divides
$s$ and the length of each of the quasicentralizers
 $\hat{y}_{j}(n_i)$ is at most $W_{n_i}-n+2$.
\end{lemma}

\begin{proof} We distinguish in the product
 (\ref{2n}) an initial segment of the form
$$[a_s, \,\hat{y}_{j}(n_1),\,\ldots, \,\hat{y}_{j}(n_k)]$$
 with zero sum of indices that has an initial subsegment in $L_j$.
For that we first find an integer $q$ such that $0\leq q\leq n-1$
and $ s+qj\equiv j \,({\rm mod\,} n)$; this is possible because
 $ \overline{ j}$ divides~$s$.
 Then
$$[a_s, \,\hat{y}_{j}(n_1),\,\ldots, \,\hat{y}_{j}(n_q)]\,\in\,L_j$$
 and the next  $n-1$ quasicentralizers
$\hat{y}_{j}(n_t)$ complement this initial segment to a product
with zero sum of indices. This product has the form
 $[ \underbrace{a_j,\,\ldots, \,a_j}_{n}]$ (under the Index Convention),
where the first of the  $a_j$ denotes the aforementioned product
in  $L_j$, while the other  $a_j$ are elements $\hat{y}_{j}(n_i)$.
By Lemma~\ref{l6} we freeze this product in level~$0$, that is, we
represent it as a linear combination of products in
representatives  of level~$0$ of the form $
[\underbrace{x_j(0),\,\ldots, \,x_j(0)}_{n}]$. Substituting this
expression into the product
 (\ref{2n}) we consider the  initial segment of the form
\begin{equation}\label{2n2} \big[[\underbrace{x_j(0),\,\ldots,
\,x_j(0)}_{n}],\,\hat{y}_{j}(n_{k+1})\big].
\end{equation}
By (\ref{tozh}) we move the element $\hat{y}_{j}(n_{k+1})$ to the
left in (\ref{2n2})  in view  to obtain a product with the
rightmost element $x_j(0)$.  At the first step, we get the sum
$$
\beta\big[\underbrace{x_j(0),\,\ldots,
\,x_j(0)}_{n-1},\,\hat{y}_{j}(n_{k+1}),\, x_j(0)\big]\,\,+
\alpha\big[\underbrace{x_j(0),\,\ldots\,x_j(0)}_{n-1},\,
[x_j(0),\,\hat{y}_{j}(n_{k+1} ]\big].
$$
In the second summand we move the element $
[x_j(0),\,\hat{y}_{j}(n_{k+1})]$ to the left by (\ref{tozh}) and
so on.
 As a result  we obtain a linear combination of products
of length  $n+1$ in elements  $x_j(0)$ and $\hat{y}_{j}(n_{k+1})$
with the right-most element $x_j(0)$ and the product  of the form
$$\Big[x_j(0),\big[\underbrace{x_j(0),\,[x_j(0),\ldots,\,[x_j(0)}_{n-1}\,
\hat{y}_{j}(n_{k+1})]...]\big]\Big].
$$
We represent the subproduct
$$\big[\underbrace{x_j(0),\,[x_j(0),\ldots,\,[x_j(0)}_{n-1},\,
\hat{y}_{j}(m_{k+1})]...]\big]$$  as a linear combination of
simple products of length  $n$ of the form
\begin{equation}\label{2n3}
[x_j(0),\,\ldots,\,\hat{y}_{j}(n_{k+1}),\,\ldots,\,x_j(0)].
\end{equation}
Each of them has  zero sum of indices. The sum of the lengths of
the elements involved is at most
 $(W_{n_{k+1}}-n+2)+(n-1)=W_{n_{k+1}}+1$ (representatives $x_j(0)$
are
 quasirepresentatives of length~$1$). Hence this
product is equal to $0$ by Lemma~\ref{lkvasic}.

Expanding the initial segment of length $n$ in the products with
the most-right element $x_j(0)$ by (\ref{tozh}) we get again a
linear combination of products of the form  (\ref{2n3}), which are
all trivial.
\end{proof}

\begin{lemma} [{see  \cite[Lemma 6]{ma-khu}}]\label{lbasic} Suppose that $l$ is a positive
integer~$\geq 4n-3$ and  in the product
\begin{equation}\label{lbasic1}
\big[a_{s},\,\,c_0,\,\ldots,
c_0,\,[x_{-k}(l),\,x_k(l)],\,\,c_0,\,\ldots,
c_0,\,[x_{-k}(l),\,x_k(l)]\,\,c_0,\,\ldots, c_0\big]
\end{equation}
there are at least $4n-3$
 products $[x_{-k}(l),\,x_k(l)]$ in $x$-representatives with the same pair of indices
$\pm k$, the $c_0$ are\/ $($possibly different\/$)$ products of
the form
 $[x_{-i}(0),\,x_i(0)]$
in representatives of level\/ $0$
 for\/ $($possibly different\/$)$
$i\ne 0$,  and the total number $C$ of the $c_0$-occur\-rences is
at most $(W_1-4n+3)/2$ $($on each interval between
 $a_{s}$
and  the products  $[x_{-k}(l),\,x_k(l)]$  the $c_0$ can also be
absent\/$)$. If $n_1, n_2, \ldots ,n_{4n-3}$ are arbitrary
pairwise different positive integers, all~$\leq l$, then the
product $($\ref{lbasic1}$)$
 can be represented as a linear combination of products  of the form
$$\big[v_{t},\,\hat{x}_{k}(n_{i_1}),\,\hat{x}_{k}(n_{i_2}),\,\ldots,
\,\hat{x}_{k}(n_{i_{2n-1}})\big]
$$
or
$$\big[v_{t},\,\hat{x}_{-k}(n_{i_1}),\,\hat{x}_{-k}(n_{i_2}),\,\ldots,
\,\hat{x}_{-k}(n_{i_{2n-1}})\big],
$$
where in each case there are
$2n-1$ in succession $x$-quasi\-repre\-sen\-ta\-tives with one and
the same
 index $k$ or $-k$, the levels $n_{i_1}, \ldots , n_{i_{2n-1}}$ are
pairwise distinct numbers in the set $\{ n_1, \ldots ,n_{4n-3}\},$
and the length of each of the $x$-quasirepresentatives
$\hat{x}_{\pm k}(n_{i_j})$ is at most $2C+4n-3$.
\end{lemma}

Here, as always under the Index Convention,
 the products $[x_{-k}(l),\,x_k(l)]$ can be different; the only
things that matter are the levels and the indices indicating
belonging to
 the homogeneous compo\-nents.

\begin{proof}
By Lemma~\ref{l6} we freeze the last $4n-3$ products
$[x_{-k}(l),\,x_k(l)]$ in the levels $n_1, n_2, \ldots ,n_{4n-3}$,
rename again by $a_s$ the corresponding initial segment of the
product (\ref{lbasic1}), and rewrite (\ref{lbasic1}) as a linear
combination of products of  the form
$$
\big[a_{s},c_0,\ldots, c_0,\,[x_{-k}(n_1),\,x_{k}(n_1)],\,
c_0,\ldots, c_0, \,[x_{-k }(n_{4n-3}),\,x_{k}(n_{4n-3})],\,
c_0,\ldots, c_0\big].$$ By (\ref{tozh}) we expand all the inner
brackets. In each product of the obtained linear combination there
are at least $2n-1$ pairs of consecutive elements $x_{-k}(n_i),\,
x_{k}(n_i) $ or $x_{k}(n_i),\, x_{-k}(n_i)$ with the same order of
indices~$\pm k$. We consider the case where there are at least
$2n-1$ pairs $x_{-k}(n_i),\,x_{k}(n_i)$ and hence at most $2n-2$
other ``bad" pairs $x_{k}(n_i),\,x_{-k}(n_i)$. In such a product
we successively get rid of the ``bad''
 pairs applying (\ref{tozh}) again:
$$[\ldots,\,  x_{k}(n_i),\,x_{-k}(n_i),\ldots] =\beta\,
[\ldots, x_{-k}(n_i),\,x_{k}(n_i),\ldots]\,+\, \alpha\,[\ldots,
[x_{k}(n_i),\,x_{-k}(n_i)],\ldots ].
$$
At each step the result is the sum of a product with a good pair
 replacing the bad one and a summand with the
subproduct $ [x_{k}(n_i),\,x_{-k}(n_i)]$, which we freeze in level
$0$ and thus add to the $c_0$-occur\-rences.

In the end we obtain a linear combination of products each
containing at least $2n-1$ good pairs $[x_{-k}(n_i),\,
x_{k}(n_i)]$, not containing bad pairs, and containing at most
$$(2n-2)+C \leq
(2n-2)+(W_1-4n+3)/2=(W_1-1)/2$$ elements of the form
$c_0=[x_{-i}(0),\,x_i(0)]$. In each of these products
 we transfer successively all the right
elements $x_{k}(n_i)$ of good pairs to the right aiming to collect
them at the right end of the product in the same order as they
occur in the product. The first to be transferred to the right
over some of the products $c_0=[x_{-s}(0),\,x_s(0)]$ is the
right-most of the $x_{k}(n_i)$, then the next, and so on.
Transferring $x_{k}(n_i)$ over a product $[x_{-s}(0),\,x_s(0)]$
yields an additional summand, where $x_{k}(n_i)$ is replaced by
the product $\big[x_{k}(n_i),\, \,[x_{-s}(0)\,x_s(0)]\big]$, which
is a $x$-quasirepresentative of level $n_i$ and is denoted
by~$\hat{x}_{k}(n_i)$.  In this summand this
$x$-quasirepresentative~$\hat{x}_{k}(n_i)$ takes over the role of
$x_{k}(n_i)$ and is also transferred to the right.

No other additional summands arise in this process. Indeed, the
elements $x_{k}(n_i)$ or, more generally, $\hat{x}_{k}(n_i)$ are
never transferred over one another.
 When an element $\hat{x}_{k}(n_i)$ is transferred over the
left part ${x}_{-k}(n_j)$ of another pair, the levels $n_i$ and
$n_j$ are always different. In the
 additional summand  the arising product
$[\hat{x}_{k}(n_i), \,{x}_{-k}(n_j)]$  has zero sum of indices and
the sum of the lengths of the $x$-quasirepresentatives involved is
at most $W_1+1$.  Indeed, the length of $\hat{x}_{k}(n_i)$ is at
most $2((2n-2)+C)+1=2C+4n-3\leq W_1$ (here  the elements $c_0$
contribute at most $ 2((2n-2)+C)\leq W_1-1$ to the length of
$\hat{x}_{k}(n_i)$ plus $1$ for the original element of the
transfer). Hence this subproduct is in fact equal to~$0$ by
Lemma~\ref{lkvasic} (bearing in mind that $W_1<W_i$ for all $i\geq
2$).

 The summands that had originally at least
$2n-1$ pairs of successive elements $[x_{k}(n_i),\, x_{-k}(n_i)]$
are subjected to similar transformations, with the roles of
 the
 $x_{k}(n_i)$, $\hat{x}_{k}(n_{i})$ taken over by
the $x_{-k}(n_i)$, $\hat{x}_{-k}(n_{i})$, respectively, and
``good'' and ``bad'' reversed.

The result of the collecting process described above is a linear
combination of products of the form
 \begin{equation}
 \label{14} [v_{t},\,\hat{x}_{
 k}(n_{i_1}),\,\ldots,
\,\hat{x}_{ k}(n_{i_{2n-1}})],
\end{equation}
 or
 \begin{equation}
 \label{15} [v_{t},\,\hat{x}_{-k}(n_{i_1}),\,\ldots, \hat{x}_{ -k}(n_{i_{2n-1}})],\end{equation} satisfying the
conclusion of the lemma; here $v_t$ simply denotes an initial
segment of the
 product.
\end{proof}

\begin{corollary}\label{lbasic-c} Suppose that $l$ is a positive
integer~$\geq 4n-3$ and  in the product
\begin{equation}\label{lbasic2} \big[c_0,\ldots,
c_0,\,[x_{-k}(l),\,x_k(l)],\,c_0,\ldots, c_0,\,\,
\boldsymbol{a}_{s},\,\,c_0,\ldots, c_0,\,[x_{-k}(l),\,x_k(l)],\,
c_0,\ldots, c_0\big]
\end{equation} there are at least $8n-7$
 products $[x_{-k}(l),\,x_k(l)]$ in $x$-representatives with the same pair of indices
$\pm k$, the $c_0$ are\/ $($possibly different\/$)$ products of
the form
 $[x_{-i}(0),\,x_i(0)]$
in representatives of level\/ $0$
 for\/ $($possibly different\/$)$
$i\ne 0$,  and the total number $C$ of the $c_0$-occur\-rences is
at most $(W_1-5n+5)/2$ $($on each interval between
 $a_{s}$
and  the products  $[x_{-k}(l),\,x_k(l)]$  the $c_0$ can also be
absent\/$)$. If $n_1, n_2, \ldots ,n_{4n-3}$ are arbitrary
pairwise different positive integers, all~$\leq l$, then the
product $(\ref{lbasic2})$
 can be represented as a linear combination of products  of the form
$$[v_{t},\,\hat{x}_{k}(n_{i_1}),\,\hat{x}_{k}(n_{i_2}),\,\ldots,
\,\hat{x}_{k}(n_{i_{2n-1}})]$$ or
$$[v_{t},\,\hat{x}_{-k}(n_{i_1}),\,\hat{x}_{-k}(n_{i_2}),\,\ldots,
\,\hat{x}_{-k}(n_{i_{2n-1}})],$$ where in each case there are
$2n-1$ in succession $x$-quasi\-repre\-sen\-ta\-tives with one and
the same
 index $k$ or $-k$, the levels $n_{i_1}, \ldots , n_{i_{2n-1}}$ are
pairwise distinct numbers in the set $\{ n_1, \ldots ,n_{4n-3}\},$
and the length of each of the $x$-quasirepresentatives
$\hat{x}_{\pm k}(n_{i_j})$ is at most $2C+4n-3$.
\end{corollary}

\begin{proof}
 Since the number of the products
 $[x_{-k}(l),\,x_k(l)]$ in $x$-representatives with the same indices  $\pm k$
 is at least
 $8n-7$, there are at least $4n-3$ such products either to the  left of  $a_s$ or to the right of
  $a_s$. In the case where there are at least
  $4n-3$ products $[x_{-k}(l)\,x_k(l)]$ to the right of
  $a_s$, we re-denote the initial segment $\big[c_0,\,\ldots,
\,c_0,\,[x_{-k}(l),\,x_k(l)],\,c_0,\,\ldots, \,c_0,\,\,
a_{s}\big]$ again by $a_s$  and apply  Lemma~\ref{lbasic} (it is
possible because the number of $c_0$-occurrences is at most
$(W_1-5n+5)/2\leq (W_1-4n+3)/2$ if $n\geq 2$).

 In the case where there are at least $4n-3$ products
 $[x_{-k}(l),\,x_k(l)]$ to the left of  $a_s$ we apply Lemma~\ref{lbasic} to  the initial segment preceding
 $a_s$.
  We obtain a linear
 combination of products with initial segments of the form
 (\ref{14}) or (\ref{15}).
But unlike the previous case the sum of the indices is equal to
$0$, therefore all these summands are equal to $0$ by
Lemma~\ref{l2n} since the length of each of the quasicentralizer
$\hat{x}_{\pm k}(n_{i_j})$ involved in (\ref{14}) or (\ref{15}) is
at most $2C+4n-3\leq  W_1-5n+5+4n-3=W_1-n+2\leq W_{n_j}-n+2$.
\end{proof}

\begin{lemma}[{see  \cite[Lemma 7]{ma-khu}}] \label{l-7}
If $\overline{ k}$ divides $ s$, then any product of the form
\begin{equation}
 \big[a_s,\,c_0,\ldots,
c_0,\,[x_{-k}(l),\,x_k(l)],\, c_0,\ldots, c_0,\,
[x_{-k}(l),\,x_k(l)],\,c_0,\ldots, \,c_0\big],\end{equation}
 where
there are at least $4n-3$ subproducts $[x_{-k}(l),\,x_k(l)]$ with
the same pair of indices $\pm k$, the level $l$ is at least
$4n-3$, and the $c_0$ are\/
 $($possibly different\/$)$ products
of the form $[x_{-i}(0),\,x_i(0)]$ in representatives of level\/
$0$ for\/ $($possibly different\/$)$ $i\ne 0$ $($on each interval
between
 $a_{s}$
and the products $[x_{-k}(l),\,x_k(l)]$ the $c_0$ can also be
absent\/$)$, and the number of $c_0$-occur\-rences is at most
 $(W_1-5n+5)/2$,
 is equal to~$0$.
\end{lemma}


\begin{proof} We first apply Lemma~\ref{lbasic} to our product with $1,\, 2,
\ldots ,4n-3$ as the numbers $n_1, n_2, \ldots , n_{4n-3}$. We
obtain a linear combination of products of the form
\begin{equation}\label{fcor-basic1}[v_{t},\,\hat{x}_{k}(m_1),\,\hat{x}_{k}(m_2),\,\ldots,
\,\hat{x}_{k}(m_{2n-1})], \end{equation} or
\begin{equation}\label{fcor-basic2}[v_{t},\,\hat{x}_{-k}(m_1),\,\hat{x}_{-k}(m_2),\,\ldots,
\,\hat{x}_{-k}(m_{2n-1})]\end{equation} where in each case there
are $2n-1$ in succession $x$-quasi\-repre\-sen\-ta\-tives with one
and the same index $k$ or $-k$,  the levels $m_1, \ldots ,
m_{2n-1}$ are pairwise distinct, and the lengths of the
$x$-quasirepresentatives $\hat{x}_{k}(m_i)$ are at most
 $2C+4n-3 \leq W_1-5n+5+4n-3=W_1-n+2$.

Now by Lemma~\ref{l2n} each product  (\ref{fcor-basic1}) or
(\ref{fcor-basic2}) is equal to~$0$. Indeed, the condition of
Lemma~\ref{l2n} on the lengths is satisfied. It remains to check
the divisibility condition. For each
 product arising under the transformations described the sum of
 indices
remains
 the same, equal to the sum of indices of the original product, that is,
to~$s$, and therefore is divisible by $\overline{ k}$ by
hypothesis.
 Hence the index $t$ in every product (\ref{fcor-basic1}) or
(\ref{fcor-basic2}) is divisible by $\overline{
k}=\overline{-k}=\overline{ n-k}$, since the numbers $k$ and $n-k$
are, obviously, divisible by~$\overline{ k}$.
 \end{proof}

\begin{corollary}\label{c-l-7}
If $\overline{ k}$ divides $ s$, then any product of the form
\begin{equation} [c_0,\,\ldots,
\,c_0,\,[x_{-k}(l),\,x_k(l)],\ldots,
 \boldsymbol{a_s},\ldots,\,[x_{-k}(l),\,x_k(l)],\,c_0,\,\ldots, \,c_0]
\end{equation}
 where
there are at least $8n-7$ subproducts $[x_{-k}(l),\,x_k(l)]$ with
the same pair of indices $\pm k$, the level $l$ is at least
$4n-3$, and the $c_0$ are\/
 $($possibly different\/$)$ products
of the form $[x_{-i}(0),\,x_i(0)]$ in representatives of level\/
$0$ for\/ $($possibly different\/$)$ $i\ne 0$ $($on each interval
between
 $a_{s}$
and the products $[x_{-k}(l),\,x_k(l)]$ the $c_0$ can also be
absent\/$)$, and the number of $c_0$-occur\-rences is at most
 $(W_1-5n+5)/2$,
 is equal to~$0$.
\end{corollary}


\begin{proof} The proof is analogous to that of Lemma~\ref{l-7},
but instead of  Lemma~\ref{lbasic} we should apply
Corollary~\ref{lbasic-c}.
 \end{proof}

\section{ Construction of the soluble ideal and  $zc$-elements}

Recall that $N$ is the fixed notation for the highest
 level, which is an
 $n$-bounded number determined by subsequent arguments, and the
$L_j(N)$ are the generalized centralizers constructed
in~\S\,\ref{s-rep-cen}. We set
$$
Z=_{\rm id}\left< L_{1}(N),\,L_{2}(N),\ldots ,L_{n-1}(N)\right>.$$
 This ideal generated by the subspaces  $L_j(N)$, $j\ne 0$, has $(m,n)$-bounded
 codimension in~$L$, since each subspace $L_j(N)$ has
$(m,n)$-bounded codimension in $L_j$ for $j\ne 0$, while
 $\mathrm{dim}\,L_0= m$ by hypothesis.

We shall prove that the ideal $ Z$ is soluble of $ n$-bounded
derived length and therefore is the required one. This is proved
by repeated application of Proposition~\ref{p1} to the following
sequence of subalgebras.
\vskip1ex

First we agree to choose an increasing sequence of positive
integers \linebreak  $T_1<T_2<\ldots $, all of which are
 $n$-bounded (as well as their number) but sufficiently
large compared with $n$-bounded values of certain other parameters
appearing later in the proof. In addition we assume the
differences $T_{k+1}-T_k$  to be also sufficiently large in the
same sense. This is possible because, as we shall see in
~\S\,\ref{parametry},
 the choice of those other parameters does not depend on the
$T_k$.

 Having in mind this sequence of the $T_i$ we define by induction the
 subalgebras $Z\langle i\rangle$ (the indices $i$ of the $Z\langle i\rangle$ are simply for
enumeration) and their subspaces $ Z_{k}\left< i\right> $ as
follows.

\vskip1ex

 $1^{\circ }$. For $i=1$ we set $Z\langle i\rangle=Z=_{\rm id}\left< L_{1}(N),\,L_{2}(N),\ldots
,L_{n-1}(N)\right> $ and for each $k =0,\,1,\ldots ,n-1$  define $
Z_{k}\langle 1\rangle=Z\langle 1\rangle  \cap L_k$.

\vskip1ex
$2^{\circ }$. We set
 $$Z\langle i+1\rangle =\left\langle \sum_{r=1}^{T_i}
 \Big[\big(Z_0\langle i\rangle\big)^r,\,Z_k\langle i\rangle, \,
 \big(Z_0\langle i\rangle\big)^{T_i-r}\Big]\,\, \mid\,\,k=0,\,1,\ldots ,n-1\right\rangle
$$ (the angle brackets
denote the subalgebra generated by the subspaces indicated) and
for each $k=0,\,1,\ldots ,n-1$ define $Z_k \langle
i+1\rangle=Z\langle i+1\rangle\cap L_k$.

\vskip1ex

The process of construction of the subalgebras~$Z\langle i\rangle$
continues up to a certain $n$-bounded number of steps determined
by subsequent arguments. \vskip1ex

The definition of the $Z\langle i\rangle$ is made to suit the
conclusion of Proposition~\ref{p1}: if, say, we prove
 that the  subalgebra $Z\langle i+1\rangle$ is soluble of derived length~$d$, then
 $Z\langle i\rangle$ is
also soluble of $(d,n)$-bounded derived length, since the number
$T_i$ is $n$-bounded.

\vskip1ex

We now define elements of a special form, which generate the
subspaces $Z_0\langle i\rangle.$ All of them are
 homo\-geneous products with zero
 sum of indices. They are constructed by
induction on~$i$.  Products constructed at the $i$-th step are
called
 {\it $zc$-elements of complexity~$i$}. With each
 $zc$-element of complexity~$i$ a tuple of length
 $i+1$ is associated, which consists of non-zero residues modulo~$n$
and is called the {\it type\/} of the $zc$-element.

\vskip1ex
 $1^{\circ}$ Complexity $i=0$. For an arbitrary level $U$
a {\it $zc$-element of level $U$ of complexity $0$} is any product
of the form $\big[x_{-k}(U),\,x_k(U)\big]$ in
$x$-repre\-sen\-ta\-tives of
 level $U$ for any $k\ne 0$. The {\it type\/} of this $zc$-element is
the symbol
 $(k(U))$, where $U$ indicates the level of
the $x$-repre\-sen\-ta\-tives and $k$ is the residue modulo~$n$
indicating the compo\-nents $L_{\pm k}$ that the
$x$-repre\-sen\-ta\-tives belong to.

\vskip1ex
 We now describe  the step of the inductive construction.
We first
 choose an increasing sequence of positive integers
$S_1<S_2<\ldots $, which are all $n$-bounded (as well as their
number) but sufficiently large in comparison with $n$-bounded
values of certain other parameters of the proof.
 We assume the ratios
 $S_{k+1}/S_k$ also to be sufficiently large in the same sense.
(See~\S\,\ref{parametry} for a scheme of the choice of all of
these parameters.)
 In addition, we choose a decreasing sequence of positive integers
 $C_1>C_2>\ldots $, which are all
 $n$-bounded (as well as their number) but are sufficiently
large and the differences $C_i-C_{i+1}$ are also sufficiently
large in comparison with $n$-bounded values of certain other
parameters of the proof; the choice of the $C_i$ is also depending
on subsequent arguments (see ~\S\,\ref{parametry}).

\vskip1ex

$2^{\circ}$ Complexity $i>0$. Suppose that we have already defined
$zc$-elements of complexity $i-1$ and their types $(s_{i-1}s_{i-2}
\ldots s_{1}k(U))$.  A~{\it $zc$-element of level $U$ of
complexity $i$} is any product of the form
$$
\big[\boldsymbol{u_{-s_i}},\,[\ldots\,z_0,\,c_0,\,\ldots,
c_0,\ldots,\boldsymbol{a_{s_i}},\ldots,c_0,\ldots,
c_0,\,z_0,\,\ldots ]\,\big],
$$
where $\pm s_i\ne 0$, \ the $z_0$ are (possibly different)
$zc$-elements of one and the same type $(s_{i-1}s_{i-2} \ldots
s_{1}k(U))$, the number of the $z_0$ is
 $S_i$, the $c_0$ are (possibly different)
products of the form $[x_{-j}(0)\,x_j(0)]$ for (possibly
different)
 $j\ne 0$
(on any of the intervals between $u_{s_i}$ and the $z_0$ the
elements $c_0$ can also be absent), and the total number of the
$c_0$ is at most $C_i$. The {\it type\/} of this $zc$-element is
the symbol $(s_{i}s_{i-1} \ldots s_{1}k(U))$, where the residue
$s_i$ indicating the indice of the element $a_{s_i}$ is added on
the left to the type of the element~$z_0$.

\section{ Properties of $\boldsymbol{zc}$-elements}

As we have already noted, the importance of the $zc$-elements is
in the fact that they generate subspaces $Z_0\langle i\rangle $.

\begin{lemma}[{see  \cite[Lemma 9]{ma-khu}}]\label{l-9} For each $i\geq 0$ the subspace
$Z_0\langle i+1 \rangle$ is generated by $zc$-elements of
complexity $i$ of
 types
$(s_i\ldots s_1k(N-2))$ of level $ N-2$ for all possible tuples of
residues $s_i,\ldots ,s_1, k$.\end{lemma}

\begin{proof} Induction on~$i$.

{\it Case $i=0$.} Here we must prove that for any
$s=0,\,1,\,2,\ldots $ and any indices $ k_1, k_2, \ldots k_s \in
\{ 0,\,1,\,\ldots ,n-1\}$   products
\begin{equation}\label{f-l-9} [a_{k_1},\,y_j(N),\,a_{k_{2}},\,\ldots, \,a_{k_s}]
\end{equation}
and
\begin{equation}\label{f2-l-9}[y_j(N),\,a_{k_{1}},\,\ldots, \,a_{k_s}]
\end{equation}
(under the Index Convention)
such that $j\ne 0$, $ j+k_1+\ldots +k_s\equiv 0\;({\rm mod\;}n)$
is equal to a linear combination of $zc$-elements of complexity
$0$ of level $N-2$, that is, products of the form
$[x_{-k}(N-2),\,x_k(N-2)]$ for $ k\ne 0$.

We use induction on $s$. If $s=0$ there is nothing to prove since
$ j\ne 0$ by the definition of the $L_j(N)$.

If $s=1$, this follows from Lemma~\ref{l-product}:
$[y_j(N),\,a_{-j}]=[y_j(N),\,y_{-j}(N-1)]$,\,\, and \,\,\,
$[a_{-j}\,y_j(N)]=[y_{-j}(N-1),\,y_j(N)]$, which we can freeze in
level $N-2$ to give it the required form.

For $s>1$  we can ``permute'' the elements $a_{k_u}$ situated to
the right of $y_j(N)$ in (\ref{f-l-9}) and (\ref{f2-l-9})
 modulo $$U=\sum_ {u=1}^{s-1}\sum\limits_
{i+i_1+\ldots+i_u\equiv 0\, ({\rm mod}\,
n)}\Big([L_{i_1},\,L_i(N),\, L_{i_{2}},\,\ldots,\,L_{i_u}]+
[L_i(N),\, L_{i_{1}},\,\ldots,\,L_{i_u}]\Big)
$$
 as follows:
$$
[a_{k_1},\,y_j(N),\ldots,
\boldsymbol{a_{k_{u}}},\boldsymbol{a_{k_{u+1}}},\ldots,
\,a_{k_s}]=\beta\,\big[a_{k_1}\,\,y_j(N),\ldots
\boldsymbol{a_{k_{u+1}}},\boldsymbol{a_{k_{u}}},\,\ldots,
\,a_{k_s}\big] \;({\rm mod\;}U\big)
$$
and
$$
[y_j(N), \,a_{k_1},\ldots,
\boldsymbol{a_{k_{u}}},\boldsymbol{a_{k_{u+1}}},\ldots,
\,a_{k_s}]=\beta\,\big[y_j(N),\,\,a_{k_1},\ldots
\boldsymbol{a_{k_{
u+1}}},\boldsymbol{a_{k_{u}}},\,\ldots,
\,a_{k_s}\big] \;({\rm mod\;}U\big).
$$

 By the induction hypothesis all elements  of$U$  can be expressed
in the required form. Therefore we may freely ``permute'' the
$a_{k_u}$ to the right of $y_j(N)$ in order  to express our
products in the required form.

We express every element $a_{k_u}$ with non-zero index $k_u\ne 0$
as a sum of a linear combination of $b$-representatives
$b_{k_u}(N-1)$ and a centralizer   $y_{k_u}(N-1)$ of level $N-1$
and substitute all these expressions into the products. We obtain
a linear combination of products (\ref{f-l-9}) and (\ref{f2-l-9})
\begin{equation}\label{f3-l-9}
[z_{k_1}, y_j(N),\,z_{k_2},\ldots,\,z_{k_s}]
\end{equation}
or, respectively,
\begin{equation}\label{f4-l-9}
[y_j(N),\,z_{k_1},\ldots,\,z_{k_s}],
\end{equation}
where the $z_{k_u}$ are either $b_{k_u}(N-1)$, or $y_{k_u}(N-1)$,
or $a_0$ (and the condition $ j+k_1+\cdots +k_s\equiv 0\;({\rm
mod\;}n)$ remains). If in (\ref{f3-l-9}) and (\ref{f4-l-9}) among
the $z_{k_u}$ situated to the right of $y_j(N)$ there is at least
one $y_{k_u}(N-1)$, then we ``transfer'' it to the right end of
the product (at each step
 multiplying by $\beta$),  denote by $a_{-k_u}$ the preceding
initial segment, and apply Lemma~\ref{l-product}:
$[a_{-k_u}\,y_{k_u}(N-1)]=[y_{-k_u}(N-2),\,y_{k_u}(N-1)]$, which
is  of required form after being frozen in level $N-2$. Similar
transformations should be made if $z_{k_1}=y_{k_1}(N-1)$ is a
centralizer of level $N-1$ in the product (\ref{f3-l-9}).  In this
case the element $y_j(N)$ takes over the role of $y_{k_u}(N-1)$.
We ``transfer'' it to the right end of the product (all additional
summands are in $U$ and have the required form by the induction
hypothesis),  denote by $a_{-j}$ the preceding initial segment and
apply Lemma~\ref{l-product} to $[a_{-j},\,y_{j}(N)]$. We obtain
the product $[y_{-j}(N-1),\,y_{j}(N)]$, which is of required form
after being frozen in level $N-2$.

We now  consider the  case where all the $z_{k_u}$ in
(\ref{f3-l-9}) and (\ref{f4-l-9}) are either $b_{k_u}(N-1)$, or
$a_0$. We claim that in such a product a suitable permutation of
the $ z_{k_u}$ produces an initial segment of bounded length with
zero sum of indices modulo~$n$.

For each index $u\ne 0$ that occurs less than $n^2$ times  we
 ``transfer'' all the $b_u(N-1)$ situated to the right of $y_j(N)$ (if any)
 to the left to place them right after
$y_j(N)$. Let $\hat y_{t}(N)$ denote the initial segment of length
$\leq n^3+1$ (plus 1 for the first element $z_{k_1}$ of
(\ref{f3-l-9}) formed in this way). Let $v_1,\ldots , v_r$,
\,$r\leq n-1$, be the other non-zero indices
 such that for each $v_i$ there are at least $n^2$
elements $ b_{v_i}(N-1)$ in the product. If there are no such
indices, then we must have $t=0$, since the original sum of
indices was $0$ modulo~$n$. Then  $\hat y_{t}(N)=0$  by
(\ref{basic-property}) if $W_N\geq n^3$. Let $d=(v_1,\ldots ,v_r)$
be the greatest common divisor of the $v_1, \ldots ,v_r$. Since
the sum of all indices is $ 0$ modulo $n$, the number $\overline{
d}=(d,n)$ must divide~$t$.
 By the Chinese
remainder theorem there exist integers $u_i$ such that  $ d= u
_1v_1+\cdots +u_rv_r$. Replacing the $u_i$ by their residues
modulo $n$ and changing notation we have $ d= u _1v_1+\cdots
+u_rv_r+un$, where $ u_i\in \{ 0,\,1,\ldots ,n-1\}$  for all $ i$
and $u$ is an integer. We can find an integer $ w\in \{
0,\,1,\ldots ,n-1\}$ such that $ t+w(u _1v_1+\cdots +u_rv_r)
\equiv 0\;({\rm mod\;}n)$. Indeed, this is equivalent to $t+wd
\equiv 0\;({\rm mod\;}n) $, which has the required solution
because $\overline{d}$ divides $t$, as we saw above.

 We now arrange an initial segment of the product by placing  after
$\hat y_{t}(N)$ exactly $wu_1$ elements $b_{v_1}(N-1)$, then
exactly $wu_2$ elements $ b_{v_2}(N-1)$, and so on, up to exactly
$wu_r$ elements $b_{v_r}(N-1)$. This initial segment has zero sum
of indices modulo $n$ and has length
 $\leq n^3+1+n^3$. Hence
it is equal to $0$ if $W_N\geq 2n^3$.

\vskip1ex {\it Case $i>0$}. By definition the algebra $Z\langle
i+1\rangle$ is generated by the products of the form
\begin{equation}\label{f5-l-9}
[\underbrace{z_0,\ldots, z_0}_{t}, a_{j},\,\underbrace{z_0,\ldots,
z_0}_{T_i-t}],
\end{equation}
where $t=0,\ldots, T_i$, $a_{j}\in
Z_{j} \langle i\rangle $ for various~$j$ and the $z_0$ are
(possibly different)
 elements of $Z_{0}\langle i\rangle $. By definition any
element of $Z_{0}\langle i+1\rangle $ is a linear combination of
simple products in elements of the form (\ref{f5-l-9})
 with zero sum of indices.

First suppose that the length of such a simple product is~$1$,
that is, it is an element of the form (\ref{f5-l-9}) with~$j=0$.
By the obvious inclusions
\begin{equation}\label{f7-l-9}
Z_{k}\langle
1\rangle \supseteq Z_{k}\langle  2\rangle \supseteq \dots
\supseteq Z_{k}\langle i\rangle  \supseteq Z_{k}\langle i+1\rangle
\supseteq \dots.
\end{equation}
all the $z_0$ in \ref{f5-l-9} belong also to $Z_{0}{\langle
1\rangle }$ and by the case $i=0$ proved above are equal to linear
combinations of elements of the form $[x_{-k}(N-2),\,x_k(N-2)]$
(for various $k \ne 0$). Since $ T_i$ can be chosen greater than
$S_i(n-1)$, each product in the linear combination obtained by
substitutions of these expressions for the $z_0$ has at least
$S_i$ subproducts of the form $[x_{-l}(N-2),\,x_l(N-2)]$ with one
and the same pair of indices $\pm l \ne 0$. (Here and in what
follows, the estimates of parameters are quite rough, we do not
aim to give the exact values, but rather show their existence.)
Choosing exactly $S_i$ of them we freeze in level $0$ (and length
2) the others, together with subproducts
$[x_{-k}(N-2),\,x_k(N-2)]$ with $k\ne l$ and denote them by $c_0$
adding to the $c_0$-occur\-rences. Their total number in each
product is at most
 $ T_i-S_i$. For $$W_1 \geq 2(T_i-S_i)+5n-5,
\;\;\;\;\;S_i\geq 8n-7\;\;\;\;{\rm and}\;\;\; N-2\geq 4n-3$$ the
resulting products satisfy the hypothesis of
Corollary~\ref{c-l-7}, which implies that they are all equal
to~$0$.

Thus, we only need to consider the aforementioned simple products
of length~$\geq 2$. Isolating the last element of the form
(\ref{f5-l-9}) in such a simple product and denoting by $a_{-j}$
the preceding initial segment we represent this simple product in
 the form
 \begin{equation}\label{f6-l-9}
\big[a_{-j},\,[\underbrace{z_0,\ldots,
z_0}_{t},a_j,\,\underbrace{z_0,\ldots, z_0}_{T_i-t}]\big].
\end{equation}
If $j=0$, then, as shown above, the
 subproduct $[\underbrace{z_0,\ldots,
z_0}_{t},\,a_j,\,\underbrace{z_0,\ldots, z_0}_{T_i-t}]$ is equal
to~$0$; hence we may assume that $j\ne 0$. In the product
(\ref{f6-l-9}) each of the $z_0$ by the induction hypothesis
 is a linear combination of $zc$-elements
of (possibly different) types $(t_{i-1}\ldots t_1l(N-2))$ and
therefore
 each of the $z_0$ can be assumed to be such a
$zc$-element. The number of all possible types
 $(t_{i-1}\ldots t_1l(N-2))$ is~$(n-1)^i$ and
 is $n$-bounded for $n$-bounded~$i$. If $T_i$
 are chosen to be $>S_i(n-1)^i$, then among the $z_0$ we can choose
 $S_i$ \ $zc$-elements of one and the same type
$(s_{i-1}\ldots s_1k(N-2))$. The other elements $z_0$ belong to
$Z_{0}\langle 1\rangle $ by (\ref{f7-l-9}). By the case $i=0$
(proved above) they are linear combinations of products of length
2 with zero sum of non-zero
 indices. These
products can be frozen in level $0$ and regarded as elements of
the form $c_0$ mentioned in the definition of $zc$-elements. Their
total number in each product of the linear combination obtained
after
 substitution into (\ref{f6-l-9}) does not exceed $T_i-S_i$. If we
 choose
 $C_i\geq T_i-S_i$, then the element (\ref{f6-l-9})
is a linear combination of $zc$-elements of the type
$(js_{i-1}\ldots s_1k(N-2))$. This completes the proof of the
lemma.
\end{proof}

{\bf Definition} We call the $zc$-elements of complexity $j$
occurring at the $j$-th step of the inductive construction of a
$zc$-element $h$ of the type $(s_i \ldots s_1k(H))$ and of
complexity $i \geq j$ {\it $zc$-elements of the type $(s_{j}
\ldots s_1k(H))$ embedded\/} in the $zc$-element~$h$. Thus, in $h$
there are embedded $S_i$ \ $zc$-elements of complexity $i-1$ of
the type $(s_{i-1} \ldots s_1k(H))$, in each of which there are
embedded $S_{i-1}$ \ $zc$-elements of the type $(s_{i-2} \ldots
s_1k(H))$, and so on. Altogether in
 $h$ there are embedded $ S_iS_{i-1}\ldots S_{j+1}$ \ $zc$-elements of
the type $(s_{j} \ldots s_1k(H))$.

\vskip1ex
 With a suitable choice of the parameters
$C_i$ and $S_i$ any substitution of
 $zc$-elements of some lower complexity $l<j$
instead of all embedded elements of a given complexity $j$ in a
given
 $zc$-element of complexity $i\geq j$ produces again a $zc$-element of
(lower) complexity $i-j+l$ (even if the types of the $zc$-elements
that are substituted are different). We shall, however, need only
certain quite special cases of this fact, mainly the case of
$l=0$, which we consider in the following lemma. \vskip1ex

 \begin{lemma}[{see  \cite[Lemma 10]{ma-khu}}] \label{l-10} Suppose that $h$ is a $zc$-element of type $(s_i
\ldots s_1k(H))$. If all the $zc$-elements of type $(s_{i_0}
\ldots s_1k(H))$
 embedded in $h$, where $i_0 \leq i$, are represented as
 linear combinations of products in $x$-repre\-sen\-ta\-tives of the form
 $[x_{-t_j}(T),\,x_{t_j}(T)]$, $j=1,\,2, \ldots $, then $h$
 can be represented as a linear combination of
$zc$-elements of the types $(s_i \ldots s_{i_0+1}t_j(T))$ of
complexity $i-i_0$ for the same
 numbers~$t_j$, $j=1,\,2, \ldots $ .\end{lemma}

\begin{proof} Induction on $i-i_0$. For $i=i_0$ the assertion is trivial.

For $i-i_0>0$ in the $zc$-element of type $(s_i\ldots s_1k(H))$
\begin{equation}\label{f1-l-10}
\big[\boldsymbol{u_{-s_i}},\,[c_0,\ldots, z_0, c_0,\ldots, c_0,
\boldsymbol{a_{s_i}},c_0,\ldots, c_0,z_0,\ldots, c_0]\big]
\end{equation}
the $z_0$ are (possibly different)
$zc$-elements of the type $(s_{i-1}\ldots s_1k(H))$ and their
number is~$S_i$. By the induction hypothesis each of the $z_0$ is
a linear combination of $zc$-elements of the types $(s_{i-1}
\ldots s_{i_0+1}t_j(T))$ for the numbers $t_j$ given in the lemma.
After substituting these expressions
 into (\ref{f1-l-10}) we may assume that the
 element under consideration is a linear combination of products of
the form (\ref{f1-l-10}), where the $z_0$ are $zc$-elements of the
types $(s_{i-1} \ldots s_{i_0+1}t_j(T))$. Since the indices $t_j$
are non-zero residues modulo~$n$ and the number $S_i$ can be
chosen to be $>S_{i-i_0}(n-1)$, among the $zc$-elements $z_0$
there are at least $S_{i-i_0}$
 elements of one and the same type, say, $(s_{i-1}
\ldots s_{i_0+1}t_{j_0}(T))$. Choosing exactly $S_{i-i_0}$ of them
we freeze in level $0$  the others, together with those where
$t_j\ne t_{j_0}$, thus adding them to the $c_0$-occur\-rences. The
total number of
 $c_0$-occur\-rences becomes at most $C_i +S_i$. For
$C_{i-i_0}-C_i\geq S_i$ we obtain a $zc$-element of the type
$(s_{i} \ldots s_{i_0+1}t_{j_0}(T))$.
\end{proof}

\begin{lemma}\label{l-101}
 Suppose that $h$ is a $zc$-element of type $(s_i \ldots
s_1k(H))$. If all the $zc$-elements of type $(s_{i_0} \ldots
s_1k(H))$, $i_0\leq i$,
 embedded in $h$ are represented as
 linear combinations of $zc$-elements of
the types $(s_{i_0}t_j(T))$, $j=1,\,2, \ldots $, then $h$
 can be represented as a linear combination of
$zc$-elements of the types $(s_i \ldots s_{i_0}t_j(T))$ of
complexity $i-i_0+1$ for the same
 numbers~$t_j$, $j=1,\,2, \ldots $ .\end{lemma}
\begin{proof}
 We carry out an argument
analogous to the proof of Lemma \ref{l-10}. The only difference
with the proof of Lemma~\ref{l-10} is that we substitute not
products in $x$-repre\-sen\-ta\-tives of the form
 $\big[x_{-t_j}(T),\,x_{t_j}(T)\big]$ for various $t_i$, but $zc$-elements of the types
 $(s_jt_2(T ))$ for one and the same~$s_j$ with various
 $t_2$. The
conditions on the numbers $ S_i$ and $C_i$ that are required are
quite similar: $S_{i+k}/S_i\geq n$ and $C_j-C_{j+k}\geq S_{j+k}$.
\end{proof}

The following lemma is an analog of Lemma 11 in \cite{ma-khu}. The
part (a), which we shall refer as a ``modular'' part, allows to
``jump'' levels in order to skip unsuitable residues in
$zc$-elements in order to bring together equal, or dividing each
other, residues. The ``unmodular'' part (b) allows to ``collide''
coprime or ``relatively coprime'' residues.

\begin{lemma}[{see  \cite[Lemma 11]{ma-khu}}]\label{l-contraction} Any $zc$-element
\begin{equation}\label{f1-l-11}
\big[\boldsymbol{u_{-s}}, [c_0,\ldots, c_0,z_0,c_0,\ldots,
\boldsymbol{a_{s}},\ldots, c_0,z_0,c_0,\ldots, c_0]\big]
\end{equation}
 of type $(sk(H))$ and of level $H \geq 8n+1$ can be represented

{\rm (a)} as a linear combination of products of the form
$[x_{-t}(H-8n ),\, x_t(H-8n)]$ for $($possibly different\/$)$ $t$
such that $\overline{ t}$ divides $\overline{ k}$,
 and

{\rm (b)} as a linear combination of products of the form
$[x_{-r}(H-8n ),\, x_r(H-8n )]$ for $($possibly different\/$)$ $r$
such that $(\overline{ r},\overline{ k})$ divides $(\overline{ s},
\overline{ k})$ {\rm (}in the particular case when $\overline{ s}$
and $\overline{ k}$ are coprime this is equivalent to $\overline{
r}$ and $\overline{ k}$ being coprime{\rm )}.
\end{lemma}

\begin{proof} The proof of the lemma repeats virtually word-by-word the proof of Lemma~11
in~\cite{ma-khu}. We should only replace the Jacoby identity by
(\ref{tozh}).

\vskip1ex
  By expanding the inner bracket by  (\ref{tozh})
we represent the product (\ref{f1-l-11}) as a linear combination
of products of the form
\begin{equation}
\label{f2-l-11} [\boldsymbol{u_{-s}},\, c_0,\ldots,
c_0,z_0,c_0,\ldots, \,\boldsymbol{a_{s}},\,c_0,\ldots,
c_0,z_0,c_0,\ldots, c_0]\end{equation} where, recall, the $z_0$
are (possibly different) products of the form $[x_{-k}(H),x_k(H)]$
with one and the same $k$ and $H$. If $S_1$ is at least $8n-7$,
then in each product (\ref{f2-l-11})
 there are at least $4n-3$
elements $z_0$ on the right or on the left of~$a_s$. If in
(\ref{f2-l-11}) there are at least $ 4n-3$ elements $z_0$ on the
right of~$a_s$, then
 the product
 (\ref{f2-l-11}) is equal to 0 by Lemma~\ref{l-7} (since $H \geq 4n-3$ and the numbers
 $W_i$ can be chosen $\geq 2C_1+5n-5$). Hence it suffices to
consider the products (\ref{f2-l-11}) in which there are at least
$ 4n-3$ elements $z_0$ on the left of~$ a_s$ and at most $4n-4$ on
the right of~$a_s$.

We substitute into such a product (\ref{f2-l-11}) the expression
$a_s$ as a sum of a linear combination of corresponding
$b$-repre\-sen\-ta\-tives $b_s(H-4n)$ and an element $y_s(H-4n)\in
L_s(H-4n)$. Then  (\ref{f2-l-11}) is equal to the sum of a linear
combination of products
\begin{equation}\label{f4-l-11}
 [{u_{-s}}, c_0,\ldots,
c_0,z_0,c_0,\ldots, \, {b_s}(H-4n), \ldots, c_0,z_0,c_0,\ldots,
c_0].
\end{equation}
and
\begin{equation}\label{f3-l-11}
[{u_{-s}}, c_0,\ldots, c_0,z_0,c_0,\ldots, \, {y_s}(H-4n), \ldots,
c_0,z_0,c_0,\ldots, c_0].
 \end{equation}
We freeze all the elements $z_0$ on the right of $b_s(H-4n)$ and
$y_s(H-4n)$ in (\ref{f4-l-11}) and (\ref{f3-l-11}), respectively,
in the form of products of length $2$ in level $0$, thus adding
them to the $c_0$-occur\-rences. Then both in (\ref{f4-l-11}) and
in (\ref{f3-l-11}) by using (\ref{tozh}) we ``transfer'' all the
$c_0$ that are on the right of $b_s(H-4n)$ and $y_s(H-4n)$
successively to the left over the elements $b_s(H-4n)$ and
$y_s(H-4n)$, respectively:
$$
 [\ldots,  b_s(H-4n), \,c_0,\ldots]=\beta\,\big[\ldots,
 c_0, {b_s}(H-4n),\ldots]+\alpha\,\big[ \ldots,
[{b_s}(H-4n),{c_0}],\ldots, \big],
$$
$$
[ \ldots,  {y_s}(H-4n), \,{c_0},\ldots]=\beta\,\big[ \ldots,
 c_0, {y_s}(H-4n),\ldots\big]+\alpha\,\big[ \ldots,
[{y_s}(H-4n),c_0],\ldots \big].
$$
 Additional summands  have the form
$$
\alpha\,\big[u_{-s}, c_0,\ldots, c_0,z_0,c_0,\ldots, c_0, \,{\hat{
b}}_s(H-4n), \,c_0,\ldots, c_0\big],
$$
and
$$
\alpha\, \big[{u_{-s}}, c_0,\ldots, c_0,z_0,c_0,\ldots, c_0, \,
{\hat{y}}_s(H-4n), \,c_0,\ldots,c_0\big].
$$
 respectively, where $\hat{b}_s(H-4n)=[{\hat{ b}}_s(H-4n), c_0]$ is a
quasirepresentative of level $ H-4n$ and
$\hat{y}_s(H-4n)=[{\hat{y}}_s(H-4n), c_0]$ is a quasicentralizer
of the same level $H-4n$. All the $c_0$ that remain on the right
of $\hat{y}_s(H-4n)$ and
 $\hat{b}_s(H-4n)$ are also transferred over these
 elements, which take over the roles of
 ${y}_s(H-4n)$ and ${b}_s(H-4n)$, respectively.

As a result of these transfers we obtain a linear combination of
products of the form
\begin{equation} \label{f6-l-11}
\big[[u_{-s}, c_0,\ldots, c_0,z_0,c_0,\ldots, c_0],\,
\hat{b}_s(H-4n)\big]
\end{equation}
and
\begin{equation} \label{f5-l-11}
\big[[u_{-s}, c_0,\ldots, c_0,z_0,c_0,\ldots, c_0,\ldots],\,
\hat{y}_s(H-4n)]
\end{equation}
 respectively, where in both cases there are at least
$ 4n-3$  elements $z_0$ on the left of $\hat{b}_s(H-4n)$ and
$\hat{y}_s(H-4n)$, while the number of
 elements $c_0$ is at most $ C_1+4n-4$.

Products (\ref{f6-l-11}) and (\ref{f5-l-11})  are subjected to
almost identical transformations. Namely, we apply Lemma
\ref{lbasic} to the indicated initial segments of the products
(\ref{f6-l-11}) and (\ref{f5-l-11}). The difference is that in the
case of
 (\ref{f6-l-11}) we choose for the numbers $n_1, n_2,
\ldots ,n_{4n-3}$ pairwise  distinct numbers $n_i$ satisfying the
inequalities
 $H-4n<n_i<H$, and in the
case of (\ref{f5-l-11})  we choose distinct numbers $n_i$
satisfying the inequalities
 $H-8n+1 <n_i<H-4n$. This application of Lemma~\ref{lbasic} is possible if the numbers
$W_i$ are chosen to be
 $\geq 2C_1+12n-11$.

As a result, the product (\ref{f6-l-11})  becomes equal to a
linear combination of product of the form
\begin{equation}\label{f7-l-11}[\ldots, \hat{x}_{k}(n_2),\hat{x}_{k}(n_1),\,
\hat{b}_s(H-4n)]
\end{equation} and
\begin{equation}\label{f8-l-11} [\ldots, \hat{x}_{-k}(n_2),\hat{x}_{-k}(n_1),\,
\hat{b}_s(H-4n)]\end{equation} in which on the left of $
\hat{b}_s(H-4n)$ there are  $2n-1$ in succession
$x$-quasi\-repre\-sen\-ta\-tives of pairwise distinct levels in
the interval
 $(H-4n,\, H)$ with one and the same index $
k$ or $-k$. The product (\ref{f5-l-11})  becomes equal to a linear
combination of products of the form
\begin{equation}\label{f9-l-11}
[\ldots, \hat{x}_{k}(n_2),\hat{x}_{k}(n_1),\, \hat{y}_s(H-4n)]
\end{equation}
and
 \begin{equation}\label{f10-l-11}
 [\ldots \hat{x}_{-k}(n_2),\hat{x}_{-k}(n_1),\,
\hat{y}_s(H-4n)]
\end{equation}
 in which on the left of $
\hat{y}_s(H-4n)$ there are $2n-1$ in succession
$x$-quasi\-repre\-sen\-ta\-tives of pairwise distinct levels in
the interval $(H-8n,\, H-4n)$ with one and the same index $ k$ or
$-k$. The lengths of the $x$-quasi\-repre\-sen\-ta\-tives in
(\ref{f7-l-11}), (\ref{f8-l-11}), (\ref{f9-l-11}) and
(\ref{f10-l-11}) do not exceed $2(C_1+4n-4)+4n-3=2C_1+12n-11$.

\vskip1ex First we prove part (a) of the lemma for products of the
form~(\ref{f7-l-11}). In each product (\ref{f7-l-11}) we start
moving the element $\hat{b}_s(H-4n)$ to the left. At the first
step, say, we get the sum
$$
\beta\,\big[\ldots,
\hat{x}_{k}(n_2),\,\hat{b}_s(H-4n),\,\hat{x}_{k}(n_1)\big]\,+\,\alpha\,
\big[\ldots, \hat{x}_{k}(n_2),
\,[\hat{x}_{k}(n_1),\,\,\hat{b}_s(H-4n)]\big].
$$
The last entry $
\hat{x}_{k}(n_1)$ of the first summand is an
$x$-quasi\-repre\-sen\-ta\-tive of level $n_1$
 and therefore also a centralizer of level $n_1-1$ by
Lemma~\ref{l-kv-cen} (since its length  is $\leq 2C_1+12n-11$ and
the differences $W_{n_1}-W_{n_1-1}$ can be chosen to be
 $\geq 2C_1+12n-12$). Since $ n_1-1>H-8n$, then by
 Lemma~\ref{l-product} the whole first summand has the form $[y_{-k}(H-8n
),\,y_k(H-8n )]$, which becomes the required form in part (a) with
$ t=k$ after freezing in the same level. In the second summand the
subproduct $ [\hat{x}_{k}(n_1),\, \hat{b}_s(H-4n)]$ takes over the
role of the
 element $ \hat{b}_s(H-4n)$ and is also moved to the left, over the $
\hat{x}_{k}(n_i)$, $i\geq 2$. By the same arguments after $j$
steps we obtain the
 sum of the product
 \begin{equation}\label{f11-l-11}
\alpha^j\, \big[\ldots, \hat{x}_{k}(n_{j+1}),
\big[\boldsymbol{\hat{x}}_{k}(n_{j}), \,[\ldots
[\boldsymbol{\hat{x}}_{k}(n_2),[ \boldsymbol{\hat{x}}_{k}(n_1),
\boldsymbol{\hat{b}}_s(H-4n)]]]\big]\big]
\end{equation}
and a linear combination of products of the form $[y_{-k}(H-8n
),\,y_k(H-8n )\big]$, which acquire the form required in part (a)
after freezing in the same
 level.

\vskip1ex

We choose the number of steps $j$ leading to (\ref{f11-l-11}) so
that $\overline{ s+jk}=(\overline{ s},\overline{ k})$. Such an
integer $j$ satisfying $0\leq j\leq n-1$  exists by virtue of the
following lemma from \cite{ma-khu}, which states also  certain
other facts necessary for what follows.

\vskip1ex

 Recall that $\overline{m}$ denotes the greatest common
divisor $(m,n)$. Clearly, $\overline{(m,l)}=(\overline{
m},\overline{ l})$ is the
 greatest common divisor
of three integers $n$, $m$, and $l$. Furthermore, $\overline{
m\cdot l}=\overline{ m\cdot \overline{ l}}$ for any integers $m$
and $l$. For a positive integer~$d$ we introduce the special
notation $(n\backslash d)$ for the maximal divisor of $n$ that is
coprime to~$d$. More precisely, if $\overline{ d}=p_1^{k_1}\ldots
p_l^{k_l}$ is the canonical decomposition of $\overline{ d}$ into
a product of non-trivial
 prime-powers and similarly
$n=p_1^{m_1}\ldots p_l^{m_l}p_{l+1}^{m_{l+1}}\ldots p_w^{m_w}$,
where $m_i\geq k_i$ for $i=1,\ldots ,l $, then by definition
$(n\backslash d)=p_{l+1}^{m_{l+1}}\ldots p_w^{m_w}$.

\begin{lemma}[{\cite[Lemma 12]{ma-khu}}]\label{l-12} For any positive integers $k$ and $s$

{\rm (a)} there exists an integer $j_0$ in the interval $0\leq
j_0\leq n-1$ such that $\overline{ s+j_0k}=\overline{ (s,
k)(n\backslash k')}$, where $ k'= k/(s, k);$

{\rm (b)} there exists an integer $j$ in the interval $0\leq j\leq
n-1$ such that $\overline{ s+jk}=(\overline{ s},\overline{ k});$

{\rm (c)} for any $i$ the number $(\overline{ s+ik}, \overline{
k})$ is equal to $ (\overline{ s},\overline{ k})$;

{\rm (d)} if $ (\overline{ r},\overline{ k})$ divides $
(\overline{ s},\overline{ k})$, then $\overline{ r}$ divides
$\overline{ (s, k)(n\backslash k')}$, where $ k'=k/(s, k)$.
\end{lemma}

Thus, we choose $j$ as in Lemma~\ref{l-12} (b). Then the
subproduct indicated in bold type in (\ref{f11-l-11})
$$
\big[\boldsymbol{\hat{x}}_{k}(n_{j}), \,\big[\ldots
[\boldsymbol{\hat{x}}_{k}(n_2),[ \boldsymbol{\hat{x}}_{k}(n_1),
\boldsymbol{\hat{b}}_s(H-4n)]]\big]\big]
$$
 becomes an
$x$-quasi\-repre\-sen\-ta\-tive of the form $\hat{x}_t(l)$ with
$t=s+jk$ such that
 $\overline{ t}=(\overline{ s},\overline{ k})$ of level
$l=\max\{ n_1,\ldots ,n_j\}$, since all the $n_i$ are distinct and
greater than $H-4n$. Since its
 length is at most $ 2C_1+2S_1+1$ and $ W_l-W_{l-1}$ can be chosen
 to be $\geq 2C_1+2S_1$,  this is also a centralizer of
the form $ y_t(l-1)$ by Lemma~\ref{l-kv-cen}. Then by
Lemma~\ref{l-product} the product (\ref{f11-l-11}) is equal to a
product of the form $[y_{-t}(H-8n ),\,y_t(H-8n )]$ with $
\overline{ t}=(\overline{ s},\overline{ k})$, which, obviously,
divides~$\overline{ k}$. Such a product acquires the form required
in part (a) after freezing in the same level. As a result, the
product (\ref{f7-l-11})
 is equal to a linear combination of products of the form required
in part~(a).

The product of the form (\ref{f8-l-11}) is subjected to the same
transformations as (\ref{f7-l-11}) with the only difference that
the elements $\hat{x}_{k}(n_i)$ are replaced by similar
 elements $\hat{x}_{-k}(n_i)$ and Lemma~\ref{l-12}(b) is applied to the numbers
$s$ and $n-k$. The resulting
 products have the form $[x_{-t}(H-8n ),\,x_t(H-8n
)]$ with $ \overline{ t}$ dividing~$\overline{n-k}$, which
satisfies the conclusion of part~(a),
since~$\overline{n-k}=\overline{k}$.

\vskip1ex

 We now prove part (b) for products~(\ref{f7-l-11}). In
each product (\ref{f7-l-11}) we begin moving the element $
\hat{b}_s(H-4n)$ to the left. After the first step, say, we obtain
the sum
$$
\beta\,\big[\ldots
\hat{x}_{k}(n_2),\,\,\hat{b}_s(H-4n),\,\,\hat{x}_{k}(n_1)\big]\,+\,\alpha\,
\big[\ldots \,\hat{x}_{k}(n_2),\,
[\hat{x}_{k}(n_1),\,\hat{b}_s(H-4n)]\big].$$ In the first summand
we continue moving the
 element $\hat{b}_s(H-4n)$ to the left over the
 elements $ \hat{x}_{k}(n_i)$. As a result, we obtain the
sum
\begin{equation}\label{f12-l-11}
\beta^{2n-1}\,\big[\ldots \,\hat{b}_s(H-4n),
\,\hat{x}_{k}(n_{2n-1}), \ldots,
\hat{x}_{k}(n_2),\,\hat{x}_{k}(n_1)\big]\;+\;\;\;\;\;\;\;\;
$$
$$\;\;\;\;\;\;\;+\;\alpha\beta^{l-1}\, \sum_{l=1}^{2n-1}\big[\ldots \, [\hat{x}_{k}(n_l),\,\,\hat{b}_s(H-4n)],
\,\hat{x}_{k}(n_{l-1}), \ldots,
\,\hat{x}_{k}(n_1)\big].\end{equation} The first summand is equal
to $0$ by Lemma~\ref{l2n}. Indeed, under all our transformations
the sum of indices remains the same, that is, equal to~$0$
modulo~$n$. Hence the sum of indices in the initial segment of the
first summand ending with
 $\hat{b}_s(H-4n)$ is
$-(2n-1)k$, which is divisible by $\overline{ k}$. The condition
on the length in Lemma~\ref{l2n} is also satisfied if the $W_i$
are chosen  to be $\geq 2C_1+2S_1+n-1$.

In each product under the sum in (\ref{f12-l-11}) we transfer the
subproduct \linebreak $ [\hat{x}_{k}(n_l),\,\, \hat{b}_s(H-4n)]$
to the right end of the product. Together with additional summands
arising by~(\ref{tozh}) this produces a linear combination of
products of the form
\begin{equation}\label{f13-l-11}
\big[\ldots
\boldsymbol{[}\hat{x}_{k}(n_{l_1}),\,\,\hat{b}_s(H-4n),
\,\hat{x}_{k}(n_{l_2}),\ldots,
\,\hat{x}_{k}(n_{l_j})\boldsymbol{]}\big],\;\;\;\;\;\;\;\;\;\;\;\;\;\;\;j\geq
1.
\end{equation}
The subproduct indicated in (\ref{f13-l-11}) is an
$x$-quasi\-repre\-sen\-ta\-tive of level $l=\max \{ n_{l_1},\ldots
n_{l_j}\}$, since all the $n_{l_i}$ are pairwise distinct and
greater
 than~$H-4n$. Since its length is $\leq 2C_1+2S_1+1$ and the number
$W_{l}-W_{l-1}$ can be chosen to be $\geq 2C_1+2S_1$, this is also
a centralizer of level $l-1$ by Lemma~\ref{l-kv-cen}. Hence the
whole product (\ref{f13-l-11}) has the form
\begin{equation}\label{f14-l-11}
[y_{-s-jk}(H-8n ),\,y_{s+jk}(H-8n )], \;\;\;\;\;\;\; {\rm where}
\;\,j\geq 1.
\end{equation}
 By Lemma~\ref{l-12}, $(\overline{s+jk},\,\overline{k})=
(\overline{s},\,\overline{k})$. Hence the product (\ref{f14-l-11})
has the form required in part (b) of Lemma \ref{l-contraction}
after freezing in the same level;  therefore the same is true also
for (\ref{f7-l-11}).

To prove part (b) for products of the form (\ref{f8-l-11}) we
subject them to exactly the same transformations as products of
the form (\ref{f7-l-11}) with the roles of
 elements $\hat{x}_{k}(n_i)$ taken over by
 elements $\hat{x}_{-k}(n_i)$. Lemma \ref{l-12}\,(c) is then applied to the
numbers $s$ and~$n-k$. The resulting
 products have the form $[x_{-r}(H-8n
),\,x_r(H-8n )]$ for $r$ such that $(\overline{
r},\overline{n-k})=(\overline{ r},\overline{k})$ divides
$(\overline{ s},\,\overline{n-k})= (\overline{ s},\,\overline{k})
$ and therefore satisfy part (b) of the lemma.

\vskip1ex

 We now consider products of the form (\ref{f9-l-11}) and
(\ref{f10-l-11}).
 We subject them to the same transformations as
products of the form
 (\ref{f7-l-11}) and (\ref{f8-l-11}), respectively, for proving both parts (a) and
(b) of Lemma~\ref{l-contraction} for them. In the products
emerging
 subproducts of the form
$$
\big[\hat{x}_{\pm k}(n_{l_1}),\,\,\hat{b}_s(H-4n), \,
\,\hat{x}_{\pm k}(n_{l_2}),\ldots, \,\hat{x}_{\pm k}(n_{l_j})\big]
$$
are replaced by subproducts of
the form
 $$
 \big[\hat{x}_{\pm k}(n_{l_1}),\,\,\hat{y}_s(H-4n), \, \,\hat{x}_{\pm
k}(n_{l_2}),\ldots, \,\hat{x}_{\pm k}(n_{l_j})\big]
$$
(the index
$\pm k$ is either $k$ in all places, or $-k$). For
 products (\ref{f9-l-11}) and (\ref{f10-l-11}) the levels $n_{l_j}$ were chosen
to satisfy the inequalities $n_{l_j}<H-4n$; hence these
subproducts
 are also quasicentralizers of level $H-4n$ (and of bounded
length) and therefore also centralizers of level $H-8n$ by
Lemma~\ref{l-kv-cen}. The indices in all the products will be
exactly the same as in the above arguments
 for products (\ref{f7-l-11}) and (\ref{f8-l-11}). Hence, by the same arguments
(with that adjustment for the levels), products (\ref{f9-l-11})
and (\ref{f10-l-11}) will be represented
 in the form required in part~(a), as well as in the form required in
part~(b) of Lemma~\ref{l-contraction}. \end{proof}

The following lemma is a consequence of Lemma
\ref{l-contraction}\,(a).

\begin{lemma}[{see  \cite[Lemma 13]{ma-khu}}]\label{l-13} Any $zc$-element of type $(s_i\ldots s_1k(H))$ of
level $H \geq 8in+1$ can be represented as a linear combination of
products of the form $[x_{-t}(H-8in ),\,x_t(H-8in )]$ with
$($possibly different\/$)$ $t$ such that $\overline{ t}$ divides
$\overline{ k}$. \end{lemma}

\begin{proof} Induction on $i$. For $i=1$ this follows from Lemma \ref{l-contraction}(a).

For $i>1$ in a $zc$-element $h$ of the type $(s_i\ldots s_1k(H))$
$$
[\boldsymbol{u}_{-s_i},\,[c_0,\,\ldots, c_0,\,z_0,\,c_0,\,\ldots,
c_0, \boldsymbol{a}_{s_i}, c_0,\ldots, c_0,\,z_0,\, c_0,\ldots,
c_0]]
$$
the $z_0$ are
(possibly different) $zc$-elements of type $(s_{i-1}\ldots
s_1k(H))$ of level $H$ and their number is~$S_i$. By the induction
hypothesis
 each of the $z_0$ is a linear
combination of products of the form $[x_{-t_j}(H-8(i-1)n
),\,x_{t_j}(H-8(i-1)n )]$ for generally speaking different
 $ t_j$ but such that
 $\overline{ t}_j$ divides~$\overline{ k}$. By Lemma \ref{l-10}
 the $zc$-element $h$ is equal to a linear combination of
$zc$-elements of the types $(s_it_j(H-8(i-1)n ) )$ for the same
numbers~$t_j$.

 By Lemma~\ref{l-contraction}\,(a) each of these $zc$-elements is equal to
 a linear combination of products of the form $[x_{-t}(H-8in
),\,x_t(H-8in )]$ for (various) $t$ such that $\overline{ t}$
divides $\overline{ t}_{j}$ and therefore divides $\overline{ k}$.
\end{proof}

\section{ Completion of the proof of Theorem \ref{th1}}

In this section we prove Theorem \ref{th1}.  The particular case
of $m=0$ follows from Proposition \ref{p1}: there exist a function
$f(n)$ such that $L^{(f(n))}\leq
\sum_{t=0}^{m}[L_0^t,L,L_0^{m-t}]=0$ and therefore $L$ is soluble
of $n$-bounded derived length.

\vskip1ex

 To prove Theorem~\ref{th1} in the general case it is
sufficient to show that
 $Z\langle Q+1 \rangle=0$ for some $n$-bounded number~$Q$. Then by Proposition~\ref{p1} the
algebra $Z\langle Q\rangle$ is soluble of $n$-bounded derived
length, since the number $T_{Q}$ is $n$-bounded. Then by
Proposition~\ref{p1} the algebra $Z\langle Q-1\rangle $ is soluble
of $n$-bounded derived length, since the number~$ T_{Q-1}$ is
$n$-bounded, and so on, up to the solubility of $n$-bounded
derived length of the ideal~$Z\langle 1\rangle=Z$. By
Lemma~\ref{l-9} it is sufficient to prove that for
 large enough
 $n$-bounded $Q$ and for large enough
  $ n$-bounded $N$ every $zc$-element of type
$(s_Q\ldots s_1k(N-2))$ is equal to~$0$ for any non-zero
$s_Q,\ldots ,s_1,k$. In order to use induction on $\overline{ k}$
it is convenient to re-formulate this statement
 in the form
 of the following proposition.

 \vskip1ex

Let $n=p_1^{n_1}\ldots p_w^{n_w}\geq 2$ be the canonical
factorization of~$n$
 into a product of non-trivial prime-powers and $k\in\{1,\ldots, n-1\}$
 such that  $\overline{ k}=p_1^{m_1}\ldots p_w^{m_w}$, where $0\leq m_j\leq
 n_j$ for all $j=1,\ldots, w$.
 In what follows we fixe $$H(\overline{ k})\,=\,4n-3\,\,+ \,\,8n(2n-3)
\sum\limits_{i=1}^{w}m_i,$$
$$ Q(\overline{ k})\,=\,1\,\,+\,\,
(2n-3) \sum\limits_{i=1}^{w}m_i$$ and $$N= H(n)+2.$$

\begin{proposition}[{see  \cite[Proposition 2]{ma-khu}}]\label{p2}
 For $ Q\geq
Q(\overline{ k})$ any $zc$-element of type $(s_Q\ldots s_1k(H))$
of level $H\geq H(\overline{ k})$ is equal to\/~$0$ for any
non-zero~$s_Q,\ldots ,s_1,k$. \end{proposition}

Note that in view of the ``embedded'' nature of the definition of
$zc$-elements
 in
Proposition \ref{p2} it suffices to prove
 the required equality to $0$ for $
Q=Q(\overline{ k})$ and $H=H(\overline{ k})$.

\begin{proof}
We use induction on~$\overline{ k}$.
 Suppose that $ \overline{ k}=1$. Any $zc$-element of
type $(sk(4n-3))$ is a product of the form
$$
\big[\boldsymbol{u_{-s_i}},\,[\ldots\,z_0,\,c_0,\,\ldots,
c_0,\ldots,\boldsymbol{a_{s_i}},\ldots,c_0,\ldots,
c_0,\,z_0,\,\ldots ]\,\big],
$$
where the $z_0=\big[x_{-k}(4n-4),\,x_k(4n-3)\big]$ are (possibly
different) $zc$-elements of complexity~0, the number of the $z_0$
is $S_1$,  the total number of the $c_0$ is at most $C_1$. If
$S_1$ is chosen to be at least $8n-7$, and  $W_1\geq 2C_1+5n-5$,
then by Corollary~\ref{c-l-7} for $\overline{ k}=1$ the product is
equal to~$0$, since $\overline{ k}=1$ divides $s$ for any~$s$.
Hence Proposition~\ref{p2} holds for this particular case.

Now suppose that $\overline{ k}>1$. To lighten the notation we
temporary note $Q=Q(\overline{k})$, $H=H(\overline{k})$. Since the
parameters
 $s_j$
in the type $(s_Q\ldots s_1k(H))$
 are non-zero
 residues modulo~$n$ and $Q(\overline{ k})\geq 1+ 2n-3\geq n$ (for $n\geq 2$), then among $s_Q,\ldots
,s_1$ there are at least two equal:
\begin{equation}\label{f1-p2}
s_{i_1}=s_{i_2},\;\;\;\;\;{\rm where}\;\
{i_1}<{i_2}\leq n.
\end{equation}
 Then it suffices to show that a $zc$-element $h$ of type
 $(s_Q\ldots
s_1k(H))$, where $s_{i_1}=s_{i_2}$, is equal to~$0$.

 The element $
h$ has ``embedded'' structure according to the inductive
construction, at the $i_1$st step of which there are subproducts
$z_0$ that are $zc$-elements of complexity $i_1-1$ of the type
$(s_{i_1-1}\ldots s_1k(H))$. Since $i_1\leq n-1$ and therefore
$H\geq 4n-3+8n(2n-3)\geq 1+8n(i_1-1)$, by Lemma~\ref{l-13} all
these $zc$-elements of type $(s_{i_1-1}\ldots s_1k(X))$ are equal
to linear combinations of products
$$
\big[x_{-t}\big(H-8(i_1-1)n\big),
\,x_{t}\big(H-8(i_1-1)n\big)\big]\;\;\;\;\;\;\;\; {\rm
for\;(various)}\;\,t\;\,{\rm such \;that}\; \overline{ t}\,\;{\rm
divides}\,\;\overline{ k}.
$$
By Lemma~\ref{l-10} the
$zc$-element $h$ is equal to a linear combination of $zc$-elements
of the types
\begin{equation}\label{f2-p2}
\big(s_Q\ldots s_{i_1}t(H-8(i_1-1)n )\big)\;\;\;\;\;\;\;\; {\rm
for\;(various)}\;\,t\;\,{\rm such\;that }\; \overline{ t}\,\;{\rm
divides}\,\;\overline{ k}.
\end{equation}
If $\overline{t}<\overline{k}$ and  $\overline{t }$ divides
$\overline{k}$, then $H(\overline{k})-H(\overline{t})\geq
8n(2n-3)$ and $Q(\overline{k})-Q(\overline{t})\geq 2n-3$. It
follows that $Q-i_1+1\geq Q(\overline{ t})$ and
 $H-8(i_1-1)n \geq H(\overline{ t})$ for all $i_1\leq
 n-1$ and $\overline{ t}<\overline{ k}$ such that $\overline{t }$ divides
$\overline{k}$. Therefore by the induction hypothesis $zc$-element
of type (\ref{f2-p2})  is equal to $0$ if $\overline{t}<\overline{
k}$. Hence it is sufficient to prove that $zc$-elements of types
(\ref{f2-p2}) are equal to $0$ in the case where $\overline{
t}=\overline{ k}$. To lighten the notation we re-denote $t$ again
by~$k$. We also denote $Y=H-8(i_1-1)n$, $F=Q-i_1+1$ and
 change notation for the residues in the type, so that
$s_{i_1}$ becomes $s_1$, and $ s_{i_2}$ equal to $s_{i_1}$
becomes, say,~$s_j$. Thus, it suffices to prove that $zc$-elements
$h$ of type
\begin{equation}\label{f3-p2}(s_{F}\ldots s_{1}k(Y))\end{equation} are equal to~$0$
 if
$$s_{j}=s_{1}\;\;\;\;\;\;{\rm for }\,\;j\leq n.$$

Let $z$ be a $zc$-element of the type $(s_{1}k(Y))$. It is easy to
verify that  $Y=H-8(i_1-1)n\geq 8n+1$ for $i_1\leq n-1$. By Lemma
\ref{l-contraction}\,(a) applied to $z$
 we obtain an expression of
$z$ as a linear combination of products of the form
 $$[x_{-t}(Y-8n ),\,\, x_{t}(Y-8n )] \;\;\;\;\;\;\; {\rm
for\;(various)}\;\,t\;\,{\rm such\;that }\; \overline{ t}\,\;{\rm
divides}\,\;\overline{ k}.$$ On the other hand, by
Lemma~\ref{l-contraction}\,(b), $z$ is equal to a linear
combination of products of the form
$$[x_{-r}(Y-8n ),\, x_{r}(Y-8n )] \;\;\;\;\;\;\;{\rm for
\;(various)}\;\, r\;\,{\rm such\;that}\;\,(\overline{
r},\overline{ k})\;\,{\rm divides}\;\,\,(\overline{
s}_{1},\overline{ k}).$$ Hence by Lemma \ref{l-10} we obtain that
any $zc$-element $a$ of the type $(s_{j-1}\ldots s_{1}k(Y))$ is
equal,
 on the one hand,
to a linear combination of $zc$-elements of the types
\begin{equation}\label{f4-p2}
(s_{j-1}\ldots s_{2}t(Y-8n ))\;\;\;\;\;\; {\rm
for\;(various)}\;\,t\;\,{\rm such\;that }\; \overline{ t}\,\;{\rm
divides}\,\;\overline{ k},
\end{equation}
and, on the other hand,
to a linear combination of $zc$-elements of the types
\begin{equation}\label{f5-p2}
(s_{j-1}\ldots s_{2}r(Y-8n )) \;\;\;\;\;\;\;{\rm
for\;(various)}\;\, r\;\,{\rm such\;that }\;\,(\overline{
r},\overline{ k})\;\,{\rm divides}\;\,(\overline{
s}_{1},\overline{ k}).
\end{equation}
Since $j \leq n$,  the level
$Y-8n$ is at least
 $8(j-2)n+1$. Hence we can apply Lemma~\ref{l-13} to each summand of linear
combinations of $zc$-elements of types (\ref{f4-p2}) and
(\ref{f5-p2}).
 As a result, any $zc$-element $a$ of the type $(s_{j-1}\ldots
s_{1}k(Y))$ can be represented, on the one hand, as a linear
combination of products of the ``modular'' form
\begin{equation}\label{f6-p2}
[x_{-t_1}(Y-8(j-1)n ),\,\, x_{t_1}(Y-8(j-1)n )] \;\;\;\;\;\;\;
{\rm for \;(various)}\;\,t_1\;\,{\rm such\;that }\; \overline{
t}_1\,\;{\rm divides}\,\;\overline{ k}.
\end{equation}
(Clearly, if $\overline{ t}_1$ divides $ \overline{ t}$ which
divides $\overline{ k}$, then $\overline{ t}_1$ also
divides~$\overline{ k}$.) On the other hand, such an element $a$
is equal to a linear combination of products of the ``unmodular''
form
\begin{equation}\label{f7-p2}
\begin{array}{cc}[x_{-r_1}(Y-8(j-1)n
),\,\, x_{r_1}(Y-8(j-1)n )] \vspace{1.5ex}\\ {\rm
for\;(various)}\;\, r_1\;\,{\rm such\;that }\;\,(\overline{
r}_1,\overline{ k})\;\,{\rm divides}\;\,\,(\overline{
s}_{1},\overline{ k}).\end{array}
\end{equation}
(If $\overline{ r}_1$ divides $r$ in (\ref{f5-p2}), for which
$(\overline{ r},\overline{ k})$ divides $(\overline{
s}_{1},\overline{ k})$, then $(\overline{ r}_1,\overline{ k})$
also divides $(\overline{ s}_{1},\overline{ k})$.)

 We now consider
an arbitrary $zc$-element $b$ of the type $(s_{j}\ldots
s_{1}k(Y))$. By definition,
\begin{equation}\label{f8-p2}
b=\big[\boldsymbol{u_{-s_{j}}},[\,c_0,\ldots, c_0,a, c_0,\ldots,
c_0,\,\boldsymbol{a_{s_ {j} }},c_0,\ldots, c_0,a,c_0,\ldots,
c_0]\big],
\end{equation}
where the $a$ are (possibly different) $zc$-elements of the type
$(s_{j-1}\ldots s_{1}k(Y))$ and their number is $S_j$, while the
number of $c_0$-occur\-rences is at most $C_j$. We suppose that
$S_j$ is sufficiently large. In the subproduct
$$
[c_0,\ldots, c_0,a,c_0,\ldots, c_0,\boldsymbol{a_{s_ {j}
}},c_0,\ldots, c_0,a,c_0,\ldots, c_0]
$$
we represent $A=2(4n-3)(n-1)-1$  first (from the left) elements
$a$
 as linear combinations of products of the form (\ref{f6-p2}). We obtain a linear combination of products of the form
\begin{equation}\label{f9-p2}
\begin{array}{l}
\big[c_0,\ldots,c_0,[x_{-t_1}(Y-8(j-1)n ), x_{t_1}(Y-8(j-1)n)], c_0,\ldots,c_0,\boldsymbol{a_{s_ {j} }},c_0,\ldots,\\
\ldots, c_0,[x_{-t_A}(Y-8(j-1)n), x_{t_A}(Y-8(j-1)n)]
,c_0,\ldots,c_0,\boldsymbol{a},c_0,\ldots,
\boldsymbol{a},\ldots\big],
\end{array}
\end{equation}
 where there are
sufficiently many, $S_j-A$, ``unused'' occurrences of the elements
$a$ and all the indices $t_i$ are such that $ \overline{ t}_i$
divides~$\overline{ k}$.
 In each
product~(\ref{f9-p2}) there are either $4n-3$ subproducts of the
form $\big[x_{-t_{i_0}}(Y-8(j-1)n ),\,\, x_{t_{i_0}}(Y-8(j-1)n
)\big]$ with one and the same pair of indices~$\pm t_{i_0}$ to the
right of $\boldsymbol{a_{s_ {j} }}$ or $4n-3$ such subproducts to
the left of $\boldsymbol{a_{s_ {j} }}$. In the case where there
are at least $4n-3$ such subproducts to the left of
$\boldsymbol{a_{s_ {j} }}$  we freeze  the others together with
subproducts $\big[x_{-t_i}(Y-8(j-1)n ),\, x_{t_i}(Y-8(j-1)n )\big]
$ with all other indices $t_i\ne t_{i_0}$ in level $0$ thus adding
them to $c_0$-occur\-rences.  By Lemma~\ref{l-7} applied to the
initial segment, all  the summands (\ref{f9-p2}) of this type is
trivial. (The condition on the level $Y-8(j-1)n\geq 4n-3$ holds
and  the numbers $ S_i$ and $C_i$ can be chosen such that
$C_j+S_j-A\leq (W_1-5n+5)/2$.)

If there are $4n-3$ subproducts of the form
$\big[x_{-t_{i_0}}(Y-8(j-1)n ),\,\, x_{t_{i_0}}(Y-8(j-1)n )\big]$
with one and the same pair of indices~$\pm t_{i_0}$ to the right
of $\boldsymbol{a_{s_ {j} }}$ we choose exactly $4n-3$ such
subproducts, freeze the others together with such subproducts to
the left of $\boldsymbol{a_{s_ {j} }}$ and subproducts
$\big[x_{-t_i}(Y-8(j-1)n ),\, x_{t_i}(Y-8(j-1)n )\big]$ with all
other indices $t_i\ne t_{i_0}$ in level $0$ thus adding them to
$c_0$-occur\-rences. Re-denoting $t_2=t_{i_0}$ and the initial
segment again by $a_{s_{j}}$  we obtain a product of the form
\begin{equation}\label{f10-p2}\begin{array}{c}\big[a_{s_{j}},\,c_0,\ldots,
c_0,\,\, [x_{-t_2}(Y-8(j-1)n ),\, x_{t_2}(Y-8(j-1)n )],
\,\,c_0,\ldots,
\;\;\;\;\;\;\;\;\;\;\;\;\;\;\;\;\;\;\;\;\;\\
 \\
\;\;\;\;\;\;\;\ldots c_0,\,\, [x_{-t_2}(Y-8(j-1)n ),\,\,
x_{t_2}(Y-8(j-1)n )],\, \,c_0,\ldots,  c_0,\,\,a,\,\,c_0,\ldots
\big],\end{array}\end{equation} in which there are $4n-3$
subproducts $ \big[x_{-t_2}(Y-8(j-1)n ),\, x_{t_2}(Y-8(j-1)n
)\big]$ with the same indices $ \pm t_2$ such that
$\overline{t_2}$ divides $\overline{k}$, the number of
$c_0$-occur\-rences is at most $C_j+S_j$, and, recall, there are $
S_j-A$ unused occurrences of elements~$a$.

\vskip1ex

 The core of the proof is to show that
 if $\overline{ t}_2=\overline{ k}$, then the product (\ref{f10-p2}) is equal
to~$0$. If, however, $\overline{ t}_2<\overline{ k}$, then we
shall be able to apply the induction hypothesis to those
$zc$-elements $h$ of type~(\ref{f3-p2}), where such subproducts
are embedded.

\begin{lemma} [{see  \cite[Lemma 14]{ma-khu}}] \label{l-14} If $\overline{ t}_2=\overline{ k}$, then the product $(\ref{f10-p2})$
is equal to~$0$. \end{lemma}

\begin{proof} We apply Lemma~\ref{lbasic} to an initial segment of the
product~(\ref{f10-p2}). This is possible, since $W_i$ can be
chosen  to be $ \geq 2C_j+2S_j+4n-3$, while the level $Y-8(j-1)n$
is at least $4n-3$ by definition (since $j \leq n$). As a result
we obtain a linear combination of products of the form
$$\big[v_{e},\,\,\hat{x}_{ t_2}(l_1),\,\, \hat{x}_{ t_2}(l_2),\ldots,
\hat{x}_{ t_2}(l_{2n-1}),\,\, c_0,\ldots c_0,\,a,\, c_0,\ldots,
c_0,\,\,a, \ldots \big]
$$
or
$$
\big[v_{e},\,\,\hat{x}_{- t_2}(l_1),\,\, \hat{x}_{-
t_2}(l_2),\ldots, \hat{x}_{- t_2}(l_{2n-1}),\,\, c_0,\,\ldots,
c_0,\,\,a,\,\, c_0,\ldots, c_0,\,\,a, \ldots, \big]
$$
where in
each summand all the $x$-quasi\-repre\-sen\-ta\-tives have one and
the same
 index~$t_2$ or $-t_2$ and there are $S_j-A$ occurrences of
``unused'' elements~$a$, and $v_e$ is simply an initial segment.
 The sum of
 indices of these products remains equal modulo $n$ to the sum of
indices of the original product, that is, to~$s_j$; in addition,
 $\overline{ k}=\overline{ t}_2$ and $s_j=s_1$. By
Lemma~\ref{l-12}\,(a) there is a positive integer
 $w\leq n-1$ such that $\overline{
s_j+w{t_2}}=\overline{ s_j-(n-w){t_2}}=\overline{ (s_{j},
t_2)(n\backslash t'_{2})}$, where $t'_{2}=t_{2}/(s_j, t_2)$.
Hence, by cutting off the last $d=n-w$ elements
$\hat{x}_{t_2}(l_i)$ (together with all the $ c_0$ and $a$) in
these products with indices $t_2$ and the last $d=w$ elements
$\hat{x}_{-t_2}(l_i)$ (together with all the $ c_0$ and $a$) in
products with indices~$-t_2$ we obtain in each summand of either
kind an initial segment $u_q$ with the sum of indices $q$ modulo
$n$ such that $\overline{ q}=\overline{ (s_{1}, k)(n\backslash
k')}$, where $k'=k/(s_1,k)$. As a result, the product
(\ref{f10-p2}) is a linear combination of
 products of the form
 \begin{equation}\label{f11-p2}
 \big[u_q,\underbrace{\hat{x}_{{\pm t_2}}(l_{2n-d}),
\hat{x}_{{\pm t_2}}(l_{2n-d+1}),\ldots, \hat{x}_{{\pm
t_2}}(l_{2n-1})}_{d}, c_0,\ldots,  c_0,a, c_0,\ldots, c_0,a,
\ldots \big],
\end{equation}
where all indices $\pm t_2$
are the same, either all $t_2$ or all $ -t_2$, \ $\overline{
q}=\overline{ (s_{1}, k)(n\backslash k')}$, \ $d\leq n-1$, and
there are $S_j-A$
 ``unused'' $a$-occur\-rences. We isolate for convenience a
corollary of Lemma~\ref{l-7}.

\begin{lemma}[{see  \cite[Lemma 15]{ma-khu}}]\label{l-15}
If in a product $$\big[g_{\pm {t_2}},\,c_0,\ldots,
c_0,\,a,\,c_0,\ldots,
 c_0,\,a,\, c_0,\ldots \big]
 $$
 the number of occurrences of $($possibly
different\/$)$
 elements $a$ equal to
linear combinations of products of the form\/ $(\ref{f6-p2})$ is
greater than~$(4n-3)(n-1)$, the overall length is sufficiently
small relative to
 the $W_i$, and $\overline{
k}=\overline{ t}_2$, then this product is equal to~$0$.\end{lemma}

\begin{proof} We substitute the expressions of
 the elements $a$ as
linear combinations of products of the form (\ref{f6-p2}) into our
product. Since  the number of elements~$a$ is greater
than~$(4n-3)(n-1)$, each product of the obtained linear
combination has at least~$4n-3$ subproducts
 $ \big[x_{-t_1}(Y-8(j-1)n ),\,\, x_{t_1}(Y-8(j-1)n )\big]$ with one and the
 same pair of
indices $ \pm t_1$ such that $\overline{t_1}$ divides
$\overline{k}$. Since $j\leq n$, the level $Y-8(j-1)n$ is at least
$4n-3$. Hence we can apply Lemma \ref{l-7} to each
 product of the linear combination. Indeed, in view of the
condition $\overline{ k}=\overline{ t}_2$  the divisibility
condition
 is satisfied and the numbers $W_i$, $C_i$ can be chosen such that $W_i\geq 2C_j+2S_j+5n-5$.\end{proof}

We now transform the product (\ref{f11-p2}) by transferring all
the elements $ \hat{x}_{{\pm t_2}}(l_i)$ successively to the right
over all the elements $a$ and $c_0$. First we transfer the
right-most of them, then the next, and so on. In the additional
summands arising the subproducts $ \big[\hat{x}_{{\pm
t_2}}(l_i),\, c_0\big]$ are also
 $x$-quasi\-repre\-sen\-ta\-tives and take over the role of the element being
transferred. We also transfer to the right the subproducts of the
form $$\big[\hat{x}_{\pm {t_2}}(l_i),\, a,\, c_0,\ldots,  c_0,\,
a,\ldots \big]$$ arising in the additional summands. Of course,
this will decrease the total number of occurrences of the form $
\hat{x}_{{\pm t_2}}(l_i)$. But in this case we aim not at
collecting such elements, but at ``clearing'' of them initial
segments of (\ref{f11-p2}) of the form
\begin{equation}\label{f12-p2}\big[u_q,\, c_0,\ldots,  c_0,\,a,\,
c_0,\ldots,  c_0,\,a, \ldots \big],
\end{equation}
in which there
are sufficiently many
 occurrences of elements~$a$ with only $c_0$-occur\-rences between them
 (and the number of the $c_0$ is $n$-bounded). The number of
$a$-occurrences  may also be decreasing in the process described
above. But by
 Lemma~\ref{l-15} this number can be decreased by at most
 $(n-1)(4n-3)(n-1)$ (since $d \leq
n-1$). Hence the number of $a$-occur\-rences in the initial
segments
 (\ref{f12-p2}) will be at least~$S_j-A- (n-1)^2(4n-3)$.

We now substitute into~(\ref{f12-p2}) the expressions of elements
$a$ as linear combinations of products of the ``unmodular'' form
(\ref{f7-p2}). We obtain a linear combination of products of the
form
\begin{equation}\label{f13-p2}\begin{array}{c}\big[u_{q},\,c_0,\ldots,
c_0,\,\, \big[x_{-r_1}(Y-8(j-1)n ),\,\, x_{r_1}(Y-8(j-1)n )\big],
\,\,c_0,\ldots
\;\;\;\;\;\;\;\;\;\;\;\;\;\;\;\;\;\;\;\;\;\;\\
 \\ \;\;\;\;\;\;\;\ldots c_0,\,\,
\big[x_{-r_i}(Y-8(j-1)n ),\,\, x_{r_i}(Y-8(j-1)n )\big],
\,\,c_0,\ldots \big],\end{array}\end{equation} where all the
indices $r_i$ are such that $(\overline{ r}_i,\overline{ k})$
divides $(\overline{ s}_{1},\overline{ k})$.
 If $$S_j-A- (n-1)^2(4n-3)>(n-1)(4n-3),$$
then in each product~(\ref{f13-p2}) there are $4n-3$
 subproducts with equal pairs of indices $\pm r_{i_0}$. We choose $4n-3$
subproducts $\big[x_{-r_{i_0}}(Y-8(j-1)n ),\,\,
x_{r_{i_0}}(Y-8(j-1)n )\big]$ with such indices
 and freeze the others together with subproducts with other
indices $r_i\ne  r_{i_0}$ in level~$0$ thus adding them to
 $c_0$-occur\-rences. We re-denote $r_2=r_{i_0}$ so that the resulting
products have the form
\begin{equation}\label{f14-p2}
\begin{array}{c}\big[u_{q},\,c_0,\ldots, c_0,\,
\big[x_{-r_2}(Y-8(j-1)n ),\, x_{r_2}(Y-8(j-1)n )\big],
\,c_0,\ldots
\;\;\;\;\;\;\;\;\;\;\;\;\;\;\;\;\;\;\;\;\;\\
 \\
\;\;\;\;\;\;\;\ldots c_0,\,\, \big[x_{-r_2}(Y-8(j-1)n ),\,\,
x_{r_2}(Y-8(j-1)n )\big], \,c_0,\ldots \big],
\end{array}\end{equation}
where the index $r_2$
 is such that $(\overline{ r}_2,\overline{ k})$ divides $(\overline{
s}_{1},\overline{ k})$, there are $4n-3$ subproducts $
\big[x_{-r_2}(Y-8(j-1)n ),\, x_{r_2}(Y-8(j-1)n )\big]$, and the
number of $c_0$-occur\-rences is at most $C_j+S_j$. Since $j\leq
n$, the level $Y-8(j-1)n$ is at least~$4n-3$. If $W_j\geq
2C_j+2S_j+5n-5$, all the products~(\ref{f14-p2}) are equal to~$0$
by Lemma~\ref{l-7}. Indeed, $\overline{ q}=\overline{ (s_{1},
k)(n\backslash k')}$, where $k'=k/(s_1,k)$. By
Lemma~\ref{l-12}\,(d), if $(\overline{ r}_2,\overline{ k})$
divides $(\overline{ s}_{1},\overline{ k})$, then $
\overline{r}_2$ divides $\overline{ (s_{1}, k)(n\backslash k')}$
and therefore divides~$q$.

Lemma~\ref{l-14} is proved.\end{proof}

We now complete the proof of Proposition~\ref{p2}. By
Lemma~\ref{l-14} products of the form (\ref{f10-p2}) can only be
non-zero if $\overline{ t}_2<\overline{ k}$. Freezing unused
elements $a$ in (\ref{f10-p2}) and substituting the corresponding
linear combinations into a product $b$ of the form (\ref{f8-p2})
 we obtain that any $zc$-element of type $(s_{j}\ldots
s_{1}k(Y))$ is equal to a linear combination of $zc$-elements of
the types $(s_jt_2(Y-8(j-1)n ))$ for (various) $t_2$ such that
$\overline{ t}_2<\overline{ k}$  and $\overline{ t}_2$ divides
$\overline{ k}$ (the number of occur\-rences of elements $c_0$ and
unused elements $a$ in (\ref{f10-p2}) is at most $S_j+C_j$, while
the difference $C_1-C_j$ can be chosen greater than~$S_j$). By
Lemma \ref{l-101} any
 $zc$-element $h$ of type (\ref{f3-p2}) is equal to a linear
 combination of $zc$-elements of the types $(s_{F}\ldots s_jt_2(Y-8(j-1)n
))$ for $t_2$ such that $\overline{ t}_2<\overline{ k}$ and
$\overline{ t}_2$ divides $\overline{ k}$. Since $j\leq n$ and
$i_1\leq n-1$,  for all $t_2$ such that $
\overline{t_2}<\overline{k}$ and $\overline{ t}_2$ divides
$\overline{ k}$ we have
$$
Y-8(j-1)n = H(\overline{k})-8n(i_1-1)-8n(j-1)=$$
$$=4n-3-8n(2n-3) \sum_{i=1}^w m_i-8n(i_1-1+j-1)\geq$$
$$\geq 4n-3-8n(2n-3) \sum_{i=1}^w
m_i-8n(2n-3)=4n-3-8n(2n-3)(\sum_{i=1}^w m_i -1)\geq
H(\overline{t_2})
$$
and
$$
F-j+1= Q(\overline{k})-i_1+1-j+1= 1+(2n-3)\sum_{i=1}^w
m_i-i_1+1-j+1\geq$$
 $$\geq 1+(2n-3)\sum_{i=1}^w m_i-(2n-3)=1+(2n-3)(\sum_{i=1}^w m_i-1)\geq
Q(\overline{t_2}).
$$
 By the
induction hypothesis such $zc$-elements are equal to~$0$.

Proposition~\ref{p2} and therefore Theorem~\ref{th1} are proved.
\end{proof}

\section{Choice of the parameters} \label{parametry}

In the proof of Proposition \ref{p2} and some auxiliary lemmas we
were using the following inequalities between the parameters
$W_i$, $C_i$, $S_i$, $T_i$, and $A$:
$$\begin{array}{rcl}
W_N&\geq&2n^3\;\;{\rm
\;(Lemma\; \ref{l-9})}; \vspace{0.8ex}\\
T_i/S_i&>&n-1\;\;\;{\rm for}\;\,i>1\;\;{\rm
\;(Lemma\; \ref{l-9})}; \vspace{0.8ex}\\
T_i/S_i&>&(n-1)^i\;\;{\rm \;(Lemma\; \ref{l-9})}; \vspace{0.8ex}\\
C_i&\geq&
T_i-S_i\;\;{\rm \;(Lemma\; \ref{l-9})}; \vspace{0.8ex}\\
S_i&\geq& 8n-7\;\;{\rm \;(Lemma\;
\ref{l-9},\; Lemma\;\ref{l-contraction},\; Proposition\; \ref{p2})}; \vspace{0.8ex}\\
W_1&\geq& 2(T_i-S_i)+5n-4\;\;{\rm \;(Lemma\; \; \ref{l-9})}; \vspace{0.8ex}\\
S_{i+k}/S_i&>&n-1\;\; {\rm \;(Lemma\; \ref{l-10},\; Lemma\;
\ref{l-101},\;Proposition\; \ref{p2});}
\vspace{0.8ex}\\
C_j-C_{j+k}&\geq& S_{j+k}\;\;\;{\rm for}\;\,k\geq 1\;\; {\rm
\;(Lemma\; \ref{l-10},\; Lemma\; \ref{l-101},\; Proposition\;
\ref{p2})}; \vspace{0.8ex}\end{array}$$
$$\begin{array}{rcl} W_1&\geq& 2C_1+5n-5\;\;{\rm \;(Lemma\; \ref{l-contraction},\;
Proposition\; \ref{p2})};
\vspace{0.8ex}\\
W_1&\geq& 2C_1+12n-11\;\;{\rm \;(Lemma\; \ref{l-contraction})}; \vspace{0.8ex}\\
W_l-W_{l-1}&\geq&
2C_1+12n-12\;\;{\rm \;(Lemma\; \ref{l-contraction})}; \vspace{0.8ex}\\
A&=&2(4n-3)(n-1)-1 \;\;{\rm \;(Proposition\; \ref{p2})}; \vspace{0.8ex}\\
W_l-W_{l-1}&\geq& 2C_1+2S_1\;\;{\rm \;(Lemma\; \ref{l-contraction})}; \vspace{0.8ex}\\
W_i&\geq& 2C_1+2S_1+n-1\;\;{\rm \;(Lemma\; \ref{l-contraction})}; \vspace{0.8ex}\\
W_1&\geq& 2C_j+2S_j+4n-3\;\;{\rm \;(Proposition\; \ref{p2})}; \vspace{0.8ex}\\
W_1&\geq& 2C_j+2S_j+5n-5\;\;{\rm
\;(Proposition\; \ref{p2})}; \vspace{0.8ex}\\
S_j&>&A+(n-1)^2(4n-3)+(4n-3)(n-1)\;\;\;{\rm for}\;\,j>1\;\;{\rm
\;(Proposition\; \ref{p2})}.\end{array}$$

 The number of the parameters $T_i$, $S_i$, and $C_i$ is
$Q(n)$, while the number of the parameters~$W_i$ is equal to the
highest level $N=H(n)+2$ in the construction of  generalized
centralizers. We can indeed choose all these parameters to be
 $n$-bounded and
 satisfying all these inequalities in the following
order:
 first $S_1= 8n-7$, then the $S_j$
(using the maximum of the two estimates),
 then the $T_i$ (the maximum of the two estimates), then a decreasing
sequence of the $C_i$, and finally the numbers
 $W_i$ with sufficiently large differences~$W_{i+1}-W_i$.

\section{ Completion of the proofs of main results}

 We shall need
the following  lemma.
\begin{lemma}\label{p-p}
Let $p$ be a prime number and let $\psi $ be a linear
transformation of finite order~$p^k$ of a vector space~$V$ over a
field of characteristic~$p$ the space of fixed points of which has
finite dimension~$m$. Then the dimension
 of $V$ is finite and does not exceed~$mp^k$. \end{lemma}

\begin{proof}
This is a well-known fact, the proof of which is based on
considering the Jordan form of the transformation $\psi$; see, for
example, \cite[1.7.4]{kh-book}.
\end{proof}

\noindent {\it Proof of Theorem~{\rm \ref{th2}}.\/}  We now
consider the situation under the hypothesis of Theorem \ref{th2}.
Let $L$ be a Lie type algebra over a field $\Bbb F$ and $\varphi$
 an automorphism of order $n$ of  $L$ with finite-dimensional
fixed-point subalgebra $C_L(\varphi )$ of dimension $\dim
C_L(\varphi )=m$.

First suppose that the characteristic of the field $\Bbb F$ is
equal to a prime divisor $p$ of the number $n$. Let $ \langle \psi
\rangle$ be the Sylow $p\hs$-subgroup of the group $\langle
\varphi\rangle$, and let $\langle \varphi\rangle =\langle
\psi\rangle \times \langle \chi\rangle$, where the order of $\chi$
is not divisible  by $p$. Consider the subalgebra of fixed points
$A=C_L(\chi )$. It is $\psi$-invariant and $ C_A(\psi )\subseteq
C_L(\varphi )$. Therefore, ${\rm dim\,}C_A(\psi )\leq m$, and by
Lemma~\ref{p-p}, the dimension ${\rm dim\,} A={\rm dim\,}C_L(\chi
)$ is bounded by some $(m,n)$-bounded number $u(m,n)$.
Furthermore, $\chi$ is a semisimple automorphism of the  algebra
$L$ of order $\leq n$. Thus, $L$ admits the automorphism $\chi$
and ${\rm dim\,}C_L(\chi )\leq u(m,n)$. Replacing $\varphi$ by
$\chi$ we can assume that $p$ does not divide $n$.

Let $\omega$ be a primitive $n$th root of unity. We extend the
ground field by $\omega$ and denote by $\widetilde L$ the algebra
$L\otimes _{{\Bbb F} }{\Bbb F} [\omega ]$. Then $\varphi$ induces
an automorphism of the  algebra $\widetilde L$. This automorphism
is denoted by the same letter. Its fixed-point subalgebra  has the
same dimension $m$.  Clearly, it suffices to prove
Theorem~\ref{th2}  for the algebra $\widetilde L$. Since the
characteristic of the field does not divide $n$, we have
$$\widetilde L= L_0 \oplus L_1\oplus \dots \oplus L_{n-1},
$$
where
$$L_k=\left\{ a\in \widetilde L\mid  \varphi(a)=\omega
^{k}a\right\},$$ and this decomposition is a $({\Bbb Z}/n{\Bbb
Z})$-grading, since
$$[L_s, L_t]\subseteq L_{s+t\,({\rm mod}\,n)},
$$
where $s+t$ is calculated modulo $n$.

By Theorem~\ref{th1} the algebra $\widetilde L$ has a homogeneous
soluble ideal $Z$ of finite $(m,n)$-bounded codimension and of
$n$-bounded derived length. Obviously, the ideal  $L\cap Z$ is the
sought-for soluble ideal in $L$ of finite $(m,n)$-bounded
codimension and of $n$-bounded derived length. Theorem \ref{th2}
is proved.
 \ep

\noindent {\it Proof of Theorem~{\rm \ref{th5}}.\/}
 Let  $Q$ and  $G$ be finite cyclic groups of coprime orders  $k$
 and~$n$.  Suppose that
 $L=\bigoplus_{q\in Q} L_q=\bigoplus_{g\in G} L^{(g)}$ is a $G$-graded color Lie
 superalgebra and   $L_0^{(e)}=L^{(e)}\cap L_0$ has finite dimension $m$.
Let $B=Q\times G$ be the direct product of $Q$ and $G$. The group
$B$ is cyclic of order $kn$ since $G$ and $Q$ are cyclic groups of
coprime orders. We  consider $L$ as a  $B$-graded algebra
$L=\bigoplus_{b\in B}L_{b}$ with $b=(q,g)\in Q\times G$ and
$L_b=L_q^{(g)}=L_q\cap L^{(g)}$. This $B$-graded algebra is $(\Bbb
Z/qn\Bbb Z)$-graded Lie type algebra, since
$$\big[[x,y],z\big]=\big[x,[y,z]\big]-\epsilon(p,q)\big[[x,z],y\big]$$
for $x\in L_p^{(g)}$, $y\in L_q^{(h)}$, $z\in L$. If $e$ is the
neutral element of $G$, then  the subspace $L_0^{e}$ is the
homogeneous identity component with respect to $B$-grading, hence
Theorem \ref{th1} implies that  $L$ contains a homogeneous soluble
ideal of $(n,k)$-bounded derived length and of finite
$(n,k,m)$-bounded codimension. \ep

{\it Proof of Theorem~{\rm \ref{th6}}.\/}
  Let  $Q$  be a finite cyclic group of
order  $k$. Suppose that a  color Lie
 superalgebra $L=\bigoplus_{q\in Q} L_q$  admits an automorphism $\varphi$ of finite
 order $n$ relatively prime to $k$. Recall that by definition, $\varphi$  preserves the given
$Q$-grading: $L_q^{\varphi}\subseteq L_q$ for all $q\in Q$.

First we perform exactly  the same reduction  as in the proof of
Theorem \ref{th2} to the case where the characteristic of $\Bbb F$
does not divide $n$.

Let $\omega$ be a primitive $n$th root of unity. We extend the
ground field by $\omega$ and denote by $\widetilde L$ the color
Lie superalgebra $L\otimes _{{\Bbb F} }{\Bbb F} [\omega
]=\bigoplus_{q\in Q} \widetilde{L_q}$, where
$\widetilde{L_q}=L_q\otimes _{\Bbb F }\, \F [\omega ]$. Then
$\varphi$ induces an automorphism of  $\widetilde L$. This
automorphism is denoted by the same letter. Its fixed-point
subalgebra $C_{\widetilde{L_0}}(\varphi)$ in $\widetilde{L_0}$ has
the same dimension $m$ as $C_{L_0}(\varphi)$. Clearly, it suffices
to prove Theorem~\ref{th6} for the algebra $\widetilde L$. Hence
in what follows we can assume that the ground field of $L$
contains a primitive $n$-th root of~1.

We have
$$L= L^{0} \oplus L^{1}\oplus \dots \oplus L^{n-1},
$$
where
$$L^{k}=\left\{ a\in  L\mid  \varphi(a)=\omega
^{k}a\right\},$$ and this decomposition is a $({\Bbb Z}/n{\Bbb
Z})$-grading, since
$$[L_s, L_t]\subseteq L_{s+t\,({\rm mod}\,n)},
$$
where $s+t$ is calculated modulo $n$. The color Lie superalgebra
$L$ is $(\Bbb Z/n\Bbb Z)$-graded since $L$ is a direct sum of
spaces $L^{k}$:
$$L=\bigoplus_{k\in \Bbb Z/n\Bbb Z} L^{k},\,\,\,[L_s, L_t]\subseteq L_{s+t\,({\rm mod}\,n)}$$
and   $L^{k}$ are homogeneous with respect to the
$Q$-grading, that is
$$L^{k}=\bigoplus_{q\in Q} (L^{k}\cap L_q).$$
By  hypothesis, $$\mathrm{dim}\, C_{L_0}(\varphi)= \mathrm{dim}\,
L_0^0=\mathrm{dim}\, L_0\cap L^0=m.$$ Theorem \ref{th5} implies
that  $L$ has  a homogeneous soluble
 ideal of finite $(n,k,m)$-bounded codimension
 and of $(n,k)$-bounded derived length. \ep

\end{document}